\documentclass[a4paper,reqno,10pt]{amsart}
\usepackage{t1enc} 
 
\usepackage[latin1]{inputenc}
\usepackage[english]{babel}
\usepackage{enumerate}
\usepackage{amsfonts} 
\usepackage{epsfig}
\usepackage{graphicx}
\usepackage{psfrag}
\usepackage{amsmath}
\usepackage{amsthm}
\usepackage{amssymb}
\usepackage{amscd} 
\usepackage{stmaryrd}
\usepackage{amsrefs}
\usepackage{bbm} 
\usepackage{mathrsfs} 

\usepackage{verbatim} 

\newtheoremstyle{exercise} 
  {3pt} 
  {3pt} 
  {\scriptsize\rmfamily} 
  {
\parindent} 
  {\rmfamily\scshape} 
  {.} 
  {.5em} 
  {} 

\newtheoremstyle{newplain}
  {5pt}
  {5pt}
  {\itshape}
  {}
  {\rmfamily\scshape}
  {. ---}
  {.5em}
  {}

\newtheoremstyle{newremark}
  {5pt}
  {5pt}
  {\rmfamily}
  {}
  {\rmfamily\scshape}
  {. ---}
  {.5em}
  {}

\theoremstyle{newplain}
\newtheorem*{Theorem*}{Theorem} 
\newtheorem*{Theorem2}{Theorem 2.9}
\newtheorem*{Theorem1}{Theorem 2.8}

\theoremstyle{newplain}
\newtheorem{Theorem}{Theorem}
\newtheorem{Lemma}[Theorem]{Lemma}
\newtheorem{Corollary}[Theorem]{Corollary}
\newtheorem{Proposition}[Theorem]{Proposition}

\newtheorem{Definition}[Theorem]{Definition}
\newtheorem{cons}[Theorem]{Construction}
\newtheorem{ques}[Theorem]{Question}

\theoremstyle{newremark}
\newtheorem{Remark}[Theorem]{Remark}

\theoremstyle{exercise}

\numberwithin{Theorem}{section}
\numberwithin{Exercise}{section}

\def\e{\varepsilon}

\def\At{{\mathscr{A}}_{\Theta}}
\def\Ft{\mathscr{F}_{\Theta}}
\def\F1{{\mathscr F}}

\def\l{\lambda}
\def\vt{\vartheta}

\def\pf{\begin{proof}[{\bf Proof:}]\thst}
\def\thst{\mbox{}\newline}

\newcommand{\N}{\mathbb{N}} 
\newcommand{\R}{\mathbb{R}} 

\newcommand{\Rn}{\R^n}
\newcommand{\Hm}[1]{\mathscr{H}^{#1}} 

\newcommand{\cl}[1]{\mathcal{#1}}
\newcommand{\p}[1]{\theta_{#1}}

\newcommand{\calA}{\mathscr{A}}
\newcommand{\calB}{\mathscr{B}}

\newcommand{\calF}{\mathscr{F}}

\newcommand{\calH}{\mathscr{H}}

\newcommand{\calL}{\mathscr{L}}

\newcommand{\calN}{\mathscr{N}}

\newcommand{\calS}{\mathscr{S}}

\newcommand{\bin}[2]{
\begin{pmatrix} #1 \\
#2 
\end{pF}}

\def\XXint#1#2#3{{%
\setbox0=\hbox{$#1{#2#3}{\int}$}
\vcenter{\hbox{$#2#3$}}\kern-.5\wd0}}

\newcommand{\la}{\langle}
\newcommand{\ra}{\rangle}

\renewcommand{\leq}{\leqslant}
\renewcommand{\geq}{\geqslant}

\begin{document}


\title{A classification of Reifenberg properties}
\author{Amos N. Koeller}
\address {Mathematisches Institut der Universit\"at T\"ubingen \\ 
Auf der Morgenstelle 10 \\
72076 T\"ubingen \\
Germany }
\email{akoeller@everest.mathematik.uni-tuebingen.de}
\date{\today}

\begin{abstract}
We define twelve variants of a Reifenberg's affine approximation property, which are known to be connected with the singular sets of minimal surfaces. With this motivation we investigate the regularity of the sets possessing these. We classify the properties with respect to whether $j$-dimensional Hausdorff dimension, locally finite $j$-dimensional Hausdorff measure or countable $j$-rectifiability hold. In showing that varying levels of regularity hold for the differing properties, quasi-self-similar sets, interesting in their own right, are constructed as counter examples. These counter examples also admit a connection to number theory via the use of the normal number theorem. Additionally, the intriguing result that such complexity in the counter examples is actually a necessity is shown.
\end{abstract}

\thanks{This research was supported by a stipendium from the Max Planck Institute for Gravitational Physics and Geometric Analysis (Albert Einstein Institute) 	
}
\keywords{Reifenberg, rectifiability, Hausdorff measure, and Hausdorff dimension\\ 
{\it 2010 Mathematics subject classification}: primary 28A78 and 37F35, secondary 51M15 and 53A10}

\maketitle

\section{Introduction}
A subset $A\subset \R^n$ is said to possess the $j$-dimensional $\e$-Reifenberg property if $A\subset B_{\rho_0}(x)$ for some $x\in \R^n$ and $\rho_0>0$, and for all $y\in A$ and $\rho \in (0,\rho_0]$ there exists a $j$-dimensional plane $L_{y,\rho}$ such that 
\[d_{\Hm{}}(A\cap B_{\rho}(x),L_{y,\rho}\cap B_{\rho}(x))<\e \rho,\]
where $B_{\rho}(x)$ denotes the open ball of radius $\rho$ around a point $x$ and $d_{\Hm{}}$ denotes the \emph{Hausdorff distance} defined on two sets $A, B \subset \R^n$ of by
$$d_{\Hm{}}(A,B):= \max\left\{\sup_{x\in A}\inf_{y\in B}|x-y|, \sup_{y\in B}\inf_{x\in A}|x-y|\right\}.$$ 
Reifenberg \cite{reif} showed that sets satisfying such a property are bihomeomorphic to a $j$-dimensional disc. 
Variants of Reifenberg's approximation property have since been observed in various works (see, for e.g., ~\cite{davidtoro}, \cite{davkentor}, ~\cite{davdeptor}, \cite{depkoe} or \cite{kentoro}). Most important for our motivation, however, is the work of Simon \cite{simon2}, who instrumentally used a direct variation of Reifenberg's property in showing the rectifiability of a particular class of minimal surfaces. Simon's property is stronger than Reifenberg's, implying greater regularity for his sets satisfying the property than Refenberg's in that, as is shown in this article, the sets are indeed $j$-dimensional. We note, however, that it is not immediately clear that Simon's property is the optimal property to be observed when considering minimal surfaces. 

The differing levels of regularity following from these two variants of Reifenberg's definition and the wider general interest in properties resembling Reifenberg's lead, first of all, to an investigation into further variants of Reifenberg's property. The aim being to discover what strengths of planar approximations are necessary to provide given desirable levels of regularity. Other forms of regularity that may be hoped for, in addition to the $j$-dimensionality of the approximated sets, are that the set be of locally finite measure, or indeed that the set be $j$-rectifiable.

Observing such properties is made more important in the possibility of providing stronger and more direct regularity results for the singular sets of minimal surfaces, or indeed of extending the present results to a wider class of geometric flows.

In this paper we undertake exactly this investigation. We observe twelve variants of Reifenberg's original planar approximation property. Both Reifenberg's and Simon's properties are included in this investigation.

The twelve properties are classified with respect to which regularity properties they imply. Where, as mentioned, we consider regularity in the sense of below defined concepts of dimension, locally finite measure and rectifiability. A major part of the classification is the construction of intricate sets interesting in their own right. These sets can be thought of as quasi-self similar fractal sets. The interest lies in the fact that the sets are $1$ compact linearly well approximable sets in dimension $1$ in a Reifenberg sense established in this paper, but are purely unrectifiable and not $\sigma$-$\Hm{1}$-finite. It is in proving the pure unrectifiability that a connection with number theory is estabilshed, in that the normal number theorem is used in the proof.

In the following two sections we define and discuss the properties considered and prove the positive classification results. We continue in Section 4 by constructing counter examples sufficient to complete the classification. We conclude the classification in Section 5 by demonstrating the necessary irregularity properties. It is also in section 5 that the above mentioned interesting properties of the counter examples are discussed.

Having completed the classification, we observe that Simon's property, that known to be related to the singular sets of geometric flows, leads to the most interesting classification results. Simon's property is shown to provide $j$-dimensional sets, but allow for sets that are both unrectifiable and of locally infinite measure. This interesting result is made more acute by the fact that any examples of sets satisfying Simon's property and possessing locally infinite measure must necessarily be complex at all critical points. Complex in the sense that neighbourhoods of a critical point have infinite measure but may contain no piece of Lipschitz graph intersecting the critical point. 

In Section 6 we prove the above mentioned complexity result, and then note that this also provides a criterion for showing that singular sets cannot possess such irregular properties.

\section{Reifenberg approximation properties}
In this section we introduce the list of approximation properties to be classified as well as the main Theorems that are proved in this paper. We first introduce some nomenclature.

\begin{Definition}\label{nomen}
In this paper we will use $\overline{A}$ to denote the closure of a set $A$ and $B_r(x)$ to denote a ball of radius $r$ around the point $x$. 

We write $d_{\Hm{}}$ to denote the Hausdorff distance as well as $\Hm{s}$ and $dim_{\Hm{}}$ to denote the $s$-dimensional Hausdorff measure and the Hausdorff dimension respectively. 

We use $int(A)$ to denote the interior of a set $A$, $d(A)$ to denote the diameter of a set $A$ and $d(\cdot,\cdot)$ to denote the usual distance on Euclidean spaces. We may now define, for any $A\subset \R^n$ 
$$A^r:=\{x\in \R^n : d(x,A)<r\}$$
to be the \emph{$r$-parallel body of $A$}. We write $G(n,m)$ to denote the Grassmann Manifold of $m$-dimensional subspaces of $\R^n$ and for $y\in \R^n$ we define
$$G_y(n,m):=\{G+y:G\in G(n,m)\}.$$

Finally, let $e_1,...,e_j$ be the usual orthonormal basis vectors for $\R^n$. For $j<n$ we identify $\R^j$ with the subspace of $\R^n$ spanned by $\{e_1,...,e_j\}$.

More detailed definitions and descriptions of the above concepts can be found in most geometric measure theory or fractal geometry books. See for example \cite{falconer}, \cite{federer} or \cite{mattila}.

\end{Definition}
\begin{Remark}
Although omitted in some works on fractal geometry we include the normalising constant in the definition of Hausdorff measure necessary to ensure $\Hm{n}=\calL^n$ on $\R^n$, where, here, $\calL^n$ is the $n$-dimensional Lebesgue measure.
\end{Remark}

We now define the twelve variants of Reifenberg's property that are our main objects of consideration. The subtleties of the differences between the definitions are not initially easy to get an intuitive understanding of. Some further description of the definition is provided in the Remark following the definition.
\begin{Definition}\label{defa} \thst
Let $A \subset \R^n$ be an arbitrary set and $j \in \mathbb{N}$; then

(i) $A$ has the weak $j$-dimensional $\delta$-approximation property (or \emph{$wj$ property}) for some $0<\delta <1$ if, for all $y \in A$, there is a $\rho_y > 0$ such that for all $\rho \in (0,\rho_y]$ there exists $L_{y,\rho}\in G_y(n,j)$ such that $B_{\rho}(y) \cap A \subset L_{y, \rho}^{\delta \rho}$.

(ii) $A$ has the weak $j$-dimensional $\delta$-approximation property with local $\rho_y$-uniformity (or \emph{$w\rho j$ property}) for some $0<\delta <1$ if, for all $y \in A$, there is a $\rho_y > 0$ such that for all $\rho \in (0,\rho_y]$ and all $x \in B_{\rho_y}(y) \cap A$, there exists $L_{x,\rho} \in G_x(j,n)$ such that $B_{\rho}(x) \cap A \subset L_{x, \rho}^{\delta \rho}$.

(iii)The property (i) is said to be $\rho_0$-uniform (referred to as the \emph{$w\rho_0 j$ property}), if $A$ is contained in some ball of radius $\rho_0$ and if, for every $y \in A$ and every $\rho \in (0,\rho_0]$, there exists $L_{y,\rho}\in G_y(j,n)$ such that $B_{\rho}(y) \cap A \subset L_{y, \rho}^{\delta \rho}$.

(iv) $A$ is said to have the fine weak $j$-dimensional approximation property (or \emph{$w\delta j$ property}) if A satisfies (i) for each $\delta>0$.

(v) $A$ is said to have the fine weak $j$-dimensional approximation property with local $\rho_y$-uniformity (or \emph{$w\rho\delta j$ property}) if $A$ satisfies (ii) for each $\delta > 0$.

(vi) $A$ is said to have the fine weak $j$-dimensional approximation property with \emph{$\rho_0$-uniformity} (or $w\rho_0\delta j$ property) if $A$ satisfies (iii) for each $\delta>0$.

(vii) $A$ is said to have the strong $j$-dimensional $\delta$-approximation property (or \emph{$s j$ property}) for some $0<\delta <1$ if, for each $y \in A$, there exists $L_y \in G_y(j,n)$ such that definition (i) holds with $L_{y, \rho} = L_y$ for every $\rho \in (0,\rho_y]$.

(viii) $A$ is said to have the strong $j$-dimensional $\delta$-approximation property (or \emph{$s\rho j$ property}) with local $\rho_y$-uniformity for some $0<\delta <1$ if, for all $y \in A$, there exists $L_y\in G(j,n)$ such that for all $x \in B_{\rho_y}(y)$ and all $\rho \in (0,\rho_y]$ we have $B_{\rho}(x) \cap A \subset (L_y+x)^{\delta \rho}.$ 

(ix) The property in (viii) is said to be $\rho_0$-uniform (referred to as the \emph{$s\rho_0 j$ property}) if $A$ is contained in some ball of radius $\rho_0$ and there exists $L\in G(j,n)$ such that for each $x\in A$ and $\rho\in (0,\rho_0]$
$B_{\rho}(x)\cap A \subset (L+x)^{\delta\rho}.$

(x) $A$ is said to have the fine strong $j$-dimensional approximation property (or \emph{$s\delta j$ property}) if $A$ satisfies (vii) for each $\delta>0$.

(xi) $A$ is said to have the fine strong $j$-dimensional approximation property with local $\rho_y$-uniformity (or \emph{$s\rho\delta j$ property}) if $A$ satisfies (viii) for each $\delta>0$.

(xii) $A$ is said to have the fine strong $j$-dimensional approximation property with $\rho_0$-uniformity (or \emph{$s\rho_0\delta j$ property}) if $A$ satisfies (ix) for each $\delta>0$.

Such a property as defined above will be referred to in general as a \emph{$j$-dimensional Reifenberg property} or a \emph{Reifenberg property} if the dimension is clear from the context.

For $\alpha\in \{w,s\}$, $\beta\in \{\emptyset, \rho, \rho_0\}$, $\gamma\in \{\delta\}\cup (0,1)=:\Delta$ and $j\leq n$ we write 
$R(\alpha, \beta,\gamma ; j)$ to denote the set of subsets of $\R^n$ satisfying the $\alpha\beta\gamma j$ property if $\gamma=\delta$ and to denote the set of subsets of $\Rn$ satisfying the $\alpha\beta j$ property with respect to $\gamma$ otherwise.
\end{Definition} \noindent

\begin{Remark}
It is, at first, not easy to gain an overview of the properties and their differences. It should be noted that there are three main ingredients in the definitions:

Firstly, whether the approximation is weak or strong. A weak approximation allows, for a given point, the approximating plane to vary with the reducing radius, the strong property does not.

Secondly, $\rho_y$ or $\rho_0$ uniformity. A $\rho_y$ uniform set must satisfy an appropriate approximation property at a given scale on whole neighbourhoods of points, otherwise, the appropriate approximation property need only be satisfied at the given selected point. $\rho_0$ uniformity requires that such a neighbourhood, in fact, be the entire set.

Thirdly, we differentiate between those approximations which are $\delta$-fine, and those which are not. A $\delta$-fine approximation requires that the definition holds for arbitrary $\delta>0$, otherwise we require only that the property hold for some given $\delta$. Clearly if a set is $\delta$-approximable then it is $\eta$-approximable for each $\eta \geq \delta$ and thus we are interested in arbitrarily small $\delta$ for $\delta$-fine approximations.

With this understanding of content, the property name abbreviations ($w j$ property, $s\rho j$ property etc.) and set notation $R(\alpha,\beta, \gamma; j)$ can be seen to be representative of the defining ingredients of the definition, which helps to distinguish which property is being referred to.

Clearly strong approximations are stronger than weak. Similarly $\rho_0$-uniform approximations are stronger than $\rho_y$-approximations, which in turn are stronger than approximations without radius uniformity. Moreover $\delta$-fine approximations are stronger than approximations which are not $\delta$-fine. With this in mind, it can be seen that the properties are listed in essentially ascending order of strength. It is this remark that is stated formally in Proposition $\ref{Rs}$.

Furthermore, note that, for all but Section 6 in this paper, the approximations can be taken to be two sided. That is, each occurrence of an expression of the form $B_{\rho}(y) \cap A \subset L_{y,\rho}^{\delta \rho}$ can be replaced with the appropriate expression of the form $d_{\calH}(B_{\rho}(y)\cap A, B_{\rho}(y) \cap L_{y,\rho})$. This leads to weaker regularity but stronger irregularity results. The concentration on one sidedness here is due to potential application to singular sets for which it is the one sided approximation type that arises.

Note finally that the motivating property considered by Simon in \cite{simon2} is exactly the $w\rho\delta j$ property. The property originally considered by Reifenberg in \cite{reif} can be stated as the two sided version of the $w\rho_0 j$ property.
\end{Remark}

We note formally some important relationships between the sets $R(\alpha, \beta, \gamma; j)$ following from the definition.
\begin{Proposition}\label{Rs}
Let $j, n\in \N$, $j\leq n$, $\alpha\in \{w,s\}$, $\beta\in \{\emptyset, \rho, \rho_0\}$ and $\gamma\in \Delta$.
Then
$$R(s, \beta, \gamma; j) \subset R(w, \beta, \gamma; j),$$
$$R(\alpha, \rho_0, \gamma; j) \subset R(\alpha, \rho, \gamma; j)\subset R(\alpha, \emptyset, \gamma; j),$$
$$R(\alpha, \beta, \delta; j) \subset R(\alpha, \beta, \gamma_1; j) \subset R(\alpha, \beta, \gamma_2; j) \hbox{ for }0<\gamma_1\leq \gamma_2\leq 1,  \hbox{ and}$$
$$R(\alpha, \beta, \gamma; j) \subset R(\alpha, \beta, \gamma; j+1).$$
Furthermore, if $A\subset B\in  R(\alpha, \beta, \gamma; j)$, then $A\in R(\alpha, \beta, \gamma; j)$.
\end{Proposition}

We classify the properties in Definition $\ref{defa}$ by the level of regularity they ensure. Since, as shall be seen, it is certainly possible that sets with Hausdorff dimension greater than $j$ satisfy $j$-dimensional approximations no great level of regularity can be expected. We therefore make a classification with respect to dimension, locally finite measure and rectifiability. More formally, the regularity properties with which we classify the above linear approximation properties can be stated as follows.

\begin{Definition}\label{regproperties}
$A$ is said to have \emph{strongly locally finite $\Hm{j}$ measure} (or \emph{strong local $\Hm{j}$-finiteness}) if for all compact subsets $K \subset \R^n$, $\Hm{j}(K \cap A) <\infty ,$ or equivalently, if for all $y \in \R^n$ there exists a radius $\rho_y > 0$ such that $\Hm{j}(B_{\rho_y}(y) \cap A) <\infty.$

$A$ is said to have \emph{weakly locally finite $\Hm{j}$ measure} (or \emph{weak local $\Hm{j}$-finiteness}) if for each $y \in A$ there exists a radius $\rho_y > 0$ such that $\Hm{j}(B_{\rho_y}(y) \cap A ) < \infty.$

$A$ is said to be $\Hm{j}$-$\sigma$-finite if there is a decomposition 
$$A=\bigcup_{i=1}^{\infty}A_i$$
such that $\Hm{j}(A_i)<\infty$ for each $i\in \N$.

$A$ will be said to be \emph{countably $j$-rectifiable} or simply \emph{$j$-rectifiable} if there exist $M_0\subset\R^n$ and Lipschitz functions $\{f_i\}_{i=1}^{\infty}$ satisfying 
$$A\subset M_0 \cup \bigcup_{i=1}^{\infty}f_i(\R^j)$$
and $\Hm{j}(M_0)=0$. Finally, $A$ is said to be purely countably $j$-unrectifiable if for all $j$-rectifiable subsets $F\subset A$, $\Hm{1}(F)=0$.
\end{Definition}
\begin{Remark}
We have included two types of locally finite measure because we wish to make a general classification and because both types occur in the literature. For weakly locally finite measure see, for e.g., \cite{ecker1} or \cite{evandga}. For strongly locally finite measure see, for e.g., \cite{brakke}, \cite{simon2} or \cite{simon3} . Note also that the two definitions are not the same. Consider for example $\calN:=\bigcup_{n=1}^{\infty}\R \times\left\{\frac{1}{n}\right\}$, which is weakly but not strongly locally $\Hm{1}$ finite.

Note also that we do not, as often is the case, require that a rectifiable set be of locally finite measure in either sense.
\end{Remark}
With the above regularity properties established, we can formulate the regularity questions we ask of the Reifenberg-like properties in Definition $\ref{defa}$.
\begin{ques}\label{ques1}
For each $\alpha\in \{w,s\}$, $\beta\in \{\emptyset, \rho, \rho_0\}$ and $\gamma\in \Delta$, does $P\in R(\alpha,\beta,\gamma;j)$ imply that $P$
\begin{enumerate}
\item has dimension less than or equal to $j$?,
\item has - (a) weakly or (b) strongly - locally finite $\Hm{j}$-measure?,
\item is $j$-rectifiable?.
\end{enumerate}
for each $j\in \N$?
\end{ques} \noindent
Our classification will answer these questions. We formulate such an answer by saying that the answer to a given property and question is either yes or no. For example, there are sets satisfying the $w j$ property with Hausdorff dimension greater than $j$. We therefore say that the answer to $w j$ (1) is no. With this system of classification, our first main theorem can be stated as follows.
\begin{Theorem}\label{classification}
The properties defined in Definition $\ref{defa}$ satisfy the classification given in the table below with respect to the questions given in Question $\ref{ques1}$.
\begin{equation} \label{tab2}
\begin{tabular}[h]{lccc}

\hbox{Property} &  & \hbox{Question} &     \nonumber \\
\hline
& & &  \nonumber \\
& (1) & (2) & (3) \nonumber \\
&  &  \hbox{ (a), (b)}&     \nonumber \\ \hline
& & &  \nonumber \\
$w j$ & \hbox{No} & \hbox{No, No}  & No  \nonumber \\
$w\rho j$  & \hbox{No} & \hbox{No, No} & No  \nonumber \\
$w\rho_0 j$ & \hbox{No}  & \hbox{No, No} & No  \nonumber \\
$w\delta j$ & \hbox{Yes} & \hbox{No, No}& No  \nonumber \\
$w\rho\delta j$ & \hbox{Yes} &  \hbox{No, No}& No  \nonumber \\
$w\rho_0\delta j$ & \hbox{Yes} & \hbox{Yes, Yes} & Yes  \nonumber \\
$s j$ & \hbox{Yes}  & \hbox{No, No} & Yes \nonumber \\
$s\rho j$ & \hbox{Yes} & \hbox{Yes, No}& Yes  \nonumber \\
$s\rho_0 j$ & \hbox{Yes} &  \hbox{Yes, Yes}& Yes  \nonumber \\
$s \delta j$ & \hbox{Yes} & \hbox{No, No} & Yes  \nonumber \\
$s \rho\delta j$& \hbox{Yes}  & \hbox{Yes, No} & Yes \nonumber \\
$s\rho_0\delta j$ & \hbox{Yes} & \hbox{Yes, Yes} & Yes  \nonumber \\ \hline

\end{tabular}
\end{equation}
\end{Theorem} \noindent

The proof of this theorem is a summary of results proven in Sections 3-5, and will be presented as such at the conclusion of Section 5.

In proving Theorem $\ref{classification}$, a very interesting example set arises which, despite satisfying very good linear approximation properties has very irregular measure properties. Our second main Theorem states the existence of this set. What makes the result even more interesting is the connection with number theory that is developed in proving that the set is purely countably $1$-unrectifiable, see Lemma $\ref{unrect}$, in that it makes use of the normal number theorem. 

Our second main theorem is stated below.

\begin{Theorem}\label{main2}
There is a compact set $A\subset \R^2$ with the following properties
\begin{enumerate}[(i)]
\item $A\in R(w,\rho,\delta;1)$,
\item $dim_{\Hm{}}A=1$,
\item $A$ is neither weakly nor strongly locally $\Hm{1}$-finite,
\item $A$ is not $\Hm{1}$-$\sigma$-finite and
\item $A$ is purely countably $j$-unrectifiable.
\end{enumerate}
\end{Theorem}
A proof of Theorem $\ref{main2}$ is given immediately prior to that of Theorem $\ref{classification}$ at the conclusion of Section 5.

\section{Regular Properties}
In this section we prove all of the regular properties that can be deduced from the linear approximation properties. That is, we prove all of the parts of the classification that can be answered with yes. Central to several of the regular results is the following lemma summarising results already in the literature.
\begin{Lemma}\label{lem2} Let $j,n\in \N$, $j\leq n$ and $A\subset \R^n$. Then,
\begin{enumerate}[(i)]
\item there is a function $\beta:[0,\infty) \rightarrow [0,\infty)$ with $\lim_{\eta \searrow 0}\beta(\eta)=0$ 
such that if $A\subset \R^n$ and $A \in R(w,\rho,\eta;j)$ for some $\eta\in (0,1)$, $\Hm{j+\beta(\eta)}(A)=0$, 
\item if $A \in R(s,\emptyset,\eta;j)$ for some $\eta\in (0,1)$, 
$A \subset \cup_{k=1}^{\infty}G_k$; where each $G_k$ is the graph of some Lipschitz function over some 
$j$-dimensional subspace of $\R^n$, and
\item if $A \in R(s,\rho_0,\eta;j)$ for some $\eta\in (0,1)$, 
$A \subset \cup_{k=1}^QG_k$; where $G_k$ is the graph of some Lipschitz function over some $j$-dimensional subspace 
of $\R^n$.
\end{enumerate}
\end{Lemma}\noindent
Proofs to the above results can be found in Simon \cite{simon3} (stated on page 63). The regularity results following from the above lemma are shown in the following Corollary. 

For the proof of the Corollary we need to consider the projections of sets onto planes. 
\begin{Definition}\label{projection}
By $\pi_i$ we denote the orthogonal projection from $\R^n$ onto its $i$th component. More generally, for an affine plane in $\R^n$, $L$, we denote by $\pi_L$ the orthogonal projection from $\R^n$ to $L$. Further, in $\R^2$ we will use $\R$ to denote the set $\{(x,0):x\in \R\}$.
\end{Definition}

\begin{Corollary}\label{cor1}
The answer to each of the following Definitions is yes: \newline
(1): $w\delta j$, $w\rho \delta j$, $s j$, $s \rho j$, $s \rho_0 j$, $s \delta j$, $s \rho \delta j$ and $s\rho_0\delta j$, \newline 
(2)(a): $s \rho j$, $s \rho_0 j$, $s \rho \delta j$ and $s\rho_0\delta j$, \newline
(2)(b): $s \rho_0 j$ and $s\rho_0\delta j$, \newline
(3): $s j$, $s \rho j$, $s \rho_0 j$, $s \delta j$, $s \rho \delta j$ and $s\rho_0\delta j$.
\end{Corollary}
\begin{proof}
Let $j\in \mathbb{N}$, $A\in \R(w,\emptyset,\delta;j)$ and $t>0$. Take $\eta>0$ such that $\beta(\eta)\leq t$ where $\beta:\R\rightarrow \R$ is the function given in Lemma $\ref{lem2}$ (i). Since $A\in \R(w,\emptyset,\delta;j)$  
\[\rho_{\eta ,x}:=\frac{1}{2}\sup \{r\in \R:r\in R_x\} >0\]
where, for each $x\in A$, $R_x$ denotes the set of real numbers $\rho_0>0$ such that, for all $\rho\in(0,\rho_0]$, there exists a $j$-dimensional affine plane $L_{x,\rho}$ for which $B_{\rho}(x)\cap A \subset L_{x,\rho}^{\eta\rho}$.

For each $m\in \mathbb{N}$ define $A_m:=\{x\in A:\rho_{\eta , x}\geq m^{-1}\}$. It is clear that $A=\cup_{m\in \mathbb{N}}A_m$. Further, for any $m\in A_m$, since $A_m \subset A$ and $\rho_{\eta,x}$ is bounded below in $A_m$, we see that $A_m\in R(w,\rho,\eta;j)$ with $\rho_y \geq \frac{1}{m}$ for each $y\in A_m$. It follows from Lemma $\ref{lem2}$ (i) that $\Hm{j+t}(A_m)\leq \Hm{j+\beta(\eta)}(A_m)=0$, and thus, since $m$ was arbitrary, that 
$0\leq \Hm{j+t}(A)=\sum_{m\in \mathbb{N}}\Hm{j+t}(A_m)=0$. We deduce that 
$$dimA=\inf\{s\in R:\Hm{s}(A)=0\}\leq j$$ 
and it follows that the answer to $w\delta j$ (1) is yes.

Further, since $R(w,\rho, \delta;j) \subset R(w,\emptyset, \delta;j)$ the answer to $w\rho\delta j$ (1) is also yes.

It is clear that any countable union of Lipschitz graphs over $j$-dimensional affine planes is $j$-dimensional. It thus follows from Lemma $\ref{lem2}$ $(ii)$ and $(iii)$ that the answers to $s j$ (1) and $s \rho_0 j$ (1) are yes. Similarly to the preceding paragraph, by Proposition $\ref{Rs}$ we infer that the answers to $s \rho j$, $s \delta j$, $s \rho \delta j$ and $s\rho_0\delta j$ (1) are yes. 

Suppose now that $A\in R(s,\rho_0,\eta;j)$ for some $\eta \in (0,1)$, that $ x \in \R^n$ and $\rho >0$. From Lemma $\ref{lem2}$ $(iii)$ it follows that
\[A \cap B_{\rho}(x) \subset \bigcup_{k=1}^Qg_k(\pi_{L_k}(B_{\rho}(x)))\]
where $L_k$ are $j$-dimensional affine planes and the functions $g_k$ are the Lipschitz functions over the $L_k$ which, combined, contain $A$. Denoting by Lip$f$ the Lipschitz constant of a Lipschitz function $f$ and setting $M = \max_k \{\hbox{ Lip}g_k \} < \infty$ it follows from the area formula that
\begin{eqnarray}
\Hm{j}(A \cap B_{\rho}(x)) & \leq & \sum_{k=1}^Q\Hm{j}(g_{k}(\pi_{L_k}(B_{\rho}(x))))  \leq QM\omega_j\rho_j. \nonumber
\end{eqnarray}
Since $x$ and $\rho$ were arbitrary, the answer to $s\rho_0 j$ (2) (both (a) and (b)) is yes.

We now note that should $A \in R(s,\rho,\eta;j)$ for some $\eta>0$ and $j\in \N$, then, by definition, for each $y \in A$ there is a $\rho_y > 0$ and an affine space $L_y$ such that for all $x \in B_{\rho_y}(y)$ and all $\rho \in (0,\rho_y]$ 
$B_{\rho}(x) \cap A \subset L_y^{\delta \rho}.$ 
It follows that $B_{\rho_y}(y) \cap A\in R(s,\rho_0,\eta;j)$ and thus, as in the preceding paragraph, that $\Hm{j}(B_{\rho/2}(x) \cap A) < \infty.$ This shows that the answer to $s\rho j$ (2) (a) is yes.

The remaining claims regarding answers to question (2) follow from Proposition $\ref{Rs}$.

For the answers to (3) we observe that Lemma $\ref{lem2}$ (ii) states that any set $A$ satisfying definition $s j$ can be written as a countable union of Lipschitz graphs. It follows from the definition of $j$-rectifiability that the answer to $s j$ (3) is yes. The remaining claims now follow from Proposition $\ref{Rs}$.
\end{proof} 
The remaining positive regularity results all concern the $w \rho_0 \delta j$ property. We show that, in fact, sets in $R(w,\rho_0,\delta;j)$ are subsets of a finite union of affine planes and thus have all the regularity properties being considered. 
\begin{Lemma}\label{propertyiv}
The answers to questions $w\rho_0 \delta j$ (1), (2)(a) and (b), and (3) are yes.
\end{Lemma}
\begin{proof}
Let $j \in \mathbb{N}$, $j\leq n$ and $A\subset \R^n$. Let $\rho_0$ be the radius given in the definition of the $w\rho_0\delta j$ property for which $A$ is a subset of a ball of radius $\rho_0$.

By the definition of the $w\rho_0\delta j$ property, $d(\overline{A})\leq 3\rho_0$ and therefore, $\overline{A}$ is compact. It follows that we can choose finitely many points $\{y_1,..., y_Q\}$ satisfying $A\subset \cup_{i=1}^{Q}B_{\rho_0}(y_i)$.

We show that for each $i\in \{1,...,Q\}$, $A_i:=A \cap B_{\rho_0}(y_i)$ is a subset of an affine plane.

Choose $\{\delta_k\}_{k\in \mathbb{N}}$ with $\delta_k\searrow 0$. By the definition of the $w\rho_0\delta j$ property it follows that for each $k\in \mathbb{N}$ there exists an affine plane $L_k$ containing $y_i$ such that $A_i \subset L_k^{\delta_k\rho_0}$. 

Since $G(n,j)$ is compact with the norm
$$||L_1-L_2||:=d_{\Hm{}}(B_{\rho_0}(0)\cap L_1, B_{\rho_0}(0)\cap L_2)$$
there is a subsequence of $\{k\}$ which we immediately relabel $\{k\}$ and an affine plane $L_i\in G(n,j)$ such that $(L_k-y_i) \rightarrow L_i$.

For any $\e>0$, we deduce, for sufficiently large $k$ dependent on $\e$, that $\delta_k\rho_0<\e$, $A_i\subset L_k^{\delta_k\rho_0} \subset L_i^{\delta_k\rho_0 + \e}+y_i$, and thus $A_i\subset L_k^{\delta_k\rho_0} \subset L_k^{2\e}+y_i$. It follows that $A_i\subset L_i+y_i$.

We infer that $A$ is a subset of a finite union of affine planes, from which it follows that the answers to $w\rho_0\delta j$ (1), (2)(a), (2)(b) and (3) are yes.
\end{proof}

\section{Constructing the counter examples}
By comparing Corollary $\ref{cor1}$ and Lemma $\ref{propertyiv}$ to Theorem $\ref{classification}$ it is clear that all of the positive answers have been proven. We need to show that the remaining questions can be answered in the negative. Again using Proposition $\ref{Rs}$, we need only actually construct four counter examples. These are constructed below and shown to satisfy the required $j$-dimensional approximation properties. That the constructed examples possess the required irregularity properties is then shown in the following section.
\begin{cons}\label{curlyn}
We construct
\[\calN:=\bigcup_{n=1}^{\infty}\R \times\left\{\frac{1}{n}\right\} \subset \R^2,\] and
\[\Lambda := \bigcup_{n=1}^{\infty}\bigcup_{i=1}^2 \hbox{graph}\left(\frac{(-1)^ix^2}{n}\right)\subset \R^2.\]
\end{cons}
It has already been noted that $\calN$ is weakly but not strongly locally $\Hm{1}$ finite. The above easy examples will be shown to satisfy the $s\rho_y \delta 1$ and $s \delta 1$ properties respectively. They will further be shown to be not strongly, weakly and weakly locally $\Hm{1}$ finite respectively. The examples showing that $j$-dimensional approximations can have too great a dimension or not be $j$-rectifiable are, of course, somewhat more complex. In construction $\ref{atheta}$ we construct a class of sets, variants of the Koch curve, from which we will be able to select two specific sets which, as well as satisfying the necessary measure properties, satisfy the $w\rho_0 1$ and $w\rho\delta 1$ properties respectively. 

In constructing the class of sets we use the following definition.
\begin{Definition}\label{def8}
Let 
$$L(a,b) = ((a_1,a_2),(b_1,b_2)):=\{x\in \R^2:x=t(a_1,a_2)+(1-t)(b_1,b_2), t\in [0,1]\}$$ 
be a line in $\R^2$ and $0<\Theta<\pi/4$. We define a $\Theta$-triangular cap on $L$, $T$, to be an isosceles triangle with base line $L$ and base angle $\Theta$. This definition initially leaves two options available. Should $L$ be an edge of a previously constructed triangle, $T_0$, then $T$ is chosen such that $\Hm{2}(T\cap  T_0)>0$. $T$ is otherwise chosen arbitrarily.
\end{Definition}
\begin{Definition}\label{Phi}
Let $\Psi\subset\R^{\N}$ be the set of all non-increasing sequences $\{\theta_n\}_{n=1}^{\infty}$ with 
$0< \theta_n\leq \theta_1<\pi/24$
\end{Definition}
We now define the class of sets from which our remaining counterexamples will come. It is a class of variants of, as mentioned, the Koch curve which allow the angle of the approximating triangles to reduce as the order of approximation increases.
\begin{cons}\label{atheta}
Let $\Theta =\{\theta_i\}_{i=1}^{\infty} \in \Psi$ and $T_0^{\Theta}:=T_{0,1}^{\Theta}$ be a $\theta_1$-triangular cap on $[0,1]\times \{0\}=:A_{0,1}^{\Theta}=:A_0^{\Theta}$. 

Write $A_{1,1}^{\Theta}$ and $A_{2,1}^{\Theta}$ for the two new edges with 
$$A_{1,1}^{\Theta}\cap \{(0,0)\}\not= 0.$$ 
Define $T_{1,1}^{\Theta}$ and $T_{1,2}^{\Theta}$ to be the $\theta_2$-triangular caps on $A_{1,1}^{\Theta}$ and $A_{1,2}^{\Theta}$ respectively so that $T_{1,1}^{\Theta}\cap T_{1,2}^{\Theta} \subset T_0^{\Theta}$.

Suppose now that $2^n$ $\theta_{n+1}$-triangular caps $\{T_{n,i}^{\Theta}\}_{i=1}^{2^n}$ with disjoint interiors have been constructed with $\{(0,0)\}\in T_{n,1}^{\Theta}$ and $T_{n,i}^{\Theta}\cap T_{n,j}^{\Theta}\not=\emptyset$ only if $|i-j|\leq 1$. 

Define $\{A_{n+1,i}^{\Theta}\}_{i=1}^{2^{n+1}}$ to be the $2^{n+1}$ shorter sides of the isosceles triangles $T_{n,i}^{\Theta}$ such that $\{(0,0)\} \in A_{n+1,1}^{\Theta}$ and $A_{n+1,i}^{\Theta}\cap A_{n+1,j}^{\Theta}\not= \emptyset$ only if $|i-j|\leq 1$. Define $\{T_{n+1,i}^{\Theta}\}_{i=1}^{2^{n+1}}$ to be the $2^{n+1}$ $\theta_{n+1}$-triangular caps on $A_{n+1,i}^{\Theta}$ respectively, defined so that $T_{n+1,i}^{\Theta}\subset T_{n,j}^{\Theta}$ for some $j\in \{1,...,2^n\}$.

For each $n\in \mathbb{N}$ define $A_n^{\Theta}:=\cup_{i=1}^{2^{n+1}}A_{n,i}^{\Theta}$ and $T_n^{\Theta}:=\cup_{i=1}^{2^n}T_{n,i}^{\Theta}$.
Finally, define 
$$\At:=\bigcap_{n=1}^{\infty}T_n^{\Theta}.$$
Should the $\Theta$ be clear, as will usually be the case, then the superscripts will be omitted.

We also define
$$\calA(\Psi):=\{\At:\Theta\in \Psi\}.$$
\end{cons}
\begin{Definition}\label{edgespoints}
Let $\Theta \in \Psi$. For each $n\in\mathbb{N}$, $i\in \{1,...,2^n\}$, define $E_{n,i}(\Theta)$ to be the corner points of the triangle $T_{n,i}$. Define 
$$E_n(\Theta):=\bigcup_{i=1}^{2^n}E_{n,i}(\Theta)$$
and
$$E(\Theta)=\bigcup_{n=1}^{\infty}E_n(\Theta).$$
\end{Definition}

\begin{Remark}
For $\theta_n$ constant in $n$, $\At$ is a flattened version of the Koch curve first constructed by Koch in \cite{koch} flattened to the height $(tan\theta_1)/2$. We shall find sets of the above form in $R(w,\rho_0,\eta;1)$ and $R(w,\rho,\delta;j)$; this will be sufficient to complete our classification. 

A set $\At$ for $\Theta$ a constant sequence allows the construction of sets satisfying the $w\rho_0 1$ property for given $\delta>0$ when sufficiently small $\theta$, dependent on $\delta$, is selected. 

For constant sequences $\Theta$, however, $\At\not\in R(w,\rho,\delta;j)$, as the large changes of direction at points in $E(\Theta)$ do not allow appropriate approximation for small $\delta$. It is for this reason that the sequences are allowed to tend toward zero, as doing so allows the sets to become arbitrarily flat at appropriately chosen scales. Even allowing $\theta_n$ to tend toward zero does not completely remove the problem at points in $E(\Theta)$, but the removal of $E(\Theta)$ from $\At$ in an appropriate manner discussed later allows this problem to be circumvented.

A bound on $\theta_1$ is important in the following results as it allows restriction on how quickly sets $\At$ spiral in upon themselves. Although $\pi/24$ may not be optimal, finding an optimal constant to bound $\theta_1$ is not important in this work.

It is clear that in order for a set to be $j$-dimensionally linearly approximable and of dimension greater than $j$, the set must be complicated. We have, however, seen that all sets satisfying the $w\rho\delta j$ property are $j$-dimensional. Combined with the fact, as we shall see, that $\Lambda$ is a counter example to locally $\Hm{j}$-finite measure for the $w\delta 1$ property, it is not immediately clear that an example as complicated as a set in $\calA(\Psi)$ is necessary as a counter example to the locally $\Hm{j}$-finite measure of sets in $R(w,\rho,\delta;1)$. We show that any such counter example must necessarily be complicated in Section 6 in a sense related to rectifiability (See Theorem $\ref{thm52}$).
\end{Remark}
For simplicity in working with the approximating sets, we make the following definition concerning the use of the indices. 
\begin{Definition}\label{inx}
Let $\Theta \in \Psi$. We define $i:\mathbb{N}\times \At \rightarrow \mathbb{N}$ by
$$i(n,x):=\min\{i\in \{1,...,2^n\}:x\in T_{n,i}\}.$$ 
Furthermore, we define 
$$j:\mathbb{N}\times \{1,...,2^n\}\rightarrow \{1,...,2^{n-1}\}$$ 
to be the function defined so that $T_{n,i}\subset T_{n-1,j(n,i)}$ for each $n\in \mathbb{N}$ and $i\in \{1,...,2^n\}$.
\end{Definition}

\begin{Remark}
In general there is only one $i\in \{1,...,2^n\}$ such that $x\in T_{n,i}$. There are however, two such $i$ for each $x\in E(\Theta)$ which makes taking extra measures, here taking the minimum, necessary.
\end{Remark}
Having constructed the sets that will be used to show the irregular properties relevant to Theorem $\ref{classification}$, we next show that the constructed sets do satisfy the necessary Reifenberg properties. That is, we show that $\calN$ satisfies the $s\rho\delta 1$ property therefore the $s\delta 1$ property, that $\Lambda$ satisfies the $s \delta 1$ property and therefore the $s 1$ and $w \delta 1$ properties, that sets in $\calA(\Psi)$ satisfy the $w\rho_0 1$ property and, for appropriate choices of $\Theta$, that $\At$ also satisfies the $w \rho\delta 1$ property.

In the next section we show that the constructed sets posses the necessary measure theoretic properties to be used as counter examples.

\begin{Proposition}\label{easyproperties} For each $\eta \in (0,1)$
\begin{enumerate}[(i)]
\item $\calN\in R(s,\rho,\delta;1)\subset R(s,\rho,\eta;1)$ and
\item $\Lambda \in R(s,\emptyset,\delta;1) \subset R(s,\emptyset,\eta;1)\cap R(w,\emptyset,\delta;1)$.
\end{enumerate}
\end{Proposition}
\begin{proof}
Let $y = (y_1,1/n_y)\in \calN$. Set $L_y:=\R \times \{1/n_y\}$. Then, for all $\delta >0$, $x\in \calN \cap B_{(3(n_y+1))^{-1}}(y)$, and all $\rho \in (0,(3(n_y+1))^{-1}]$, 
\[B_{\rho}(x)\cap \calN \subset L_y = L_y+x-y \subset L_y^{\delta \rho} + x-y\]
so that $\calN\in R(s,\rho,\delta;1)$ proving, together with Proposition $\ref{Rs}$, (i). 

Let $\delta_0>0$. In observing $\Lambda$ there are two types of points to consider, $x=(0,0)$ and otherwise. If $x = (x_1, x_2) \in \Lambda$ but $x \not= (0,0)$, then
$x \in \hbox{ graph}\left(\frac{sgn(x_1)sgn(x_2)  x^2}{n}\right) $
for some $n \in \mathbb{N}$ and for $r_x = \frac{|x|^2\delta_0}{4(n+1)}$
$$B_{r_x}(x) \cap \Lambda  \subset  \hbox{ graph}\left(\frac{sgn(x_1)sgn(x_2)  x^2}{n}\right).$$
Since $x^2$ is differentiable,  there is a tangent line $L_{x}$ to the graph of $sgn(x_1)sgn(x_2)x^2/n$ at $x$ and a radius, 
$r_x \geq r_{x_1} = r_{x_1}(\delta_0)>0$,  such that for all 
\[ y  \in \hbox{ graph}\left(\frac{sgn(x_1)sgn(x_2)x^2}{n}\right) \cap B_{r_{x_1}}(x)\]
\[|\pi_{L_x^{\perp}}(y) - \pi_{L_x^{\perp}}(x)| < \delta_0 |\pi_{L_x}(y) - \pi_{L_x}(x)|\]
so that $B_{r}(x) \cap \Lambda \subset L_{x}^{\delta_0 r}$ for each $r \in (0,r_{x_1}]$. 

If $x = (0,0)$, then by construction, we may choose $L_x = \R$ and note that 
for $|x| < \delta_0$
\[\frac{|x^2|}{n} = \frac{|x||x|}{n} < |x| \delta_0 \]
for each $n \in \mathbb{N}$. Thus, for each $r \in (0,\delta_0]$, $\Lambda \cap B_{r}((0,0)) \subset L_x^{r\delta_0}.$

Noting that in each case $\delta_0>0$ was arbitrarily selected, it follows that $\Lambda \in R(s,\emptyset, \delta;1)$.
\end{proof}

The proof that sets $\At\in \calA(\Psi)$ satisfy particular Reifenberg properties for appropriate $\Theta \in \Psi$ is somewhat more involved. The difficulty lies in the fact that the sets $\At$ begin to spiral in on themselves as the level of approximation via the construction increases. Spiraling is clearly an unwanted property for linear approximation. 

We get around the problem by first proving that any spiraling is not too tight. To do this we need to control the angles between the triangular caps, which we first define.

\begin{Definition}\label{cxvs}
Let $V\in G(2,1)$, $a\in \R^2$ and $0<s<1$. We define
$$C(x,V,s):=\{y\in \R^2:d(y-x,V)<sd(x,y)\}.$$
\end{Definition}

\begin{Definition}\label{parallels}
For lines $L_1, L_2 \subset \R^2$ we write $L_1||L_2$ to denote that the lines are parallel.

For $n\in \mathbb{N}$, $i\in \{1,...,2^n\}$ and $\Theta \in \Psi$ let $G_{n,i}^{\Theta}\in G(2,1)$ be the line satisfying $G_{n,i}^{\Theta}||A_{n,i}^{\Theta}$. For $z\in \R^2$ define $G_{n,i}^{\Theta,z}:=G_{n,i}^{\Theta}+z$. As usual, the superscript $\Theta$ will be suppressed if it is clear from the context. 

Also, suppose that $L$ is a line in $\R^2$ of finite length with midpoint $l$, then we use $O_L$ to denote an orthogonal isometry $O_L:L \rightarrow \R\times \{0\}$ satisfying $O_L(l) = (0,0).$
\end{Definition}

\begin{Definition}\label{def17} 
For $G\in G(2,1)$ let $G:\R^2\rightarrow \R^2$ denote the rotation satisfying $G(G)=\R$,
$$G(z)>0\Leftrightarrow \pi_2(z)>0 \hbox{ for all }z\in G \hbox{ if }G\not= \R, \hbox{ and } G(z)=z \hbox{ if }G=\R.$$

Let $A, B \subset \R^2$, $G\in G(2,1)$ is then said to divide $A$ and $B$ if there is a $g\in G$ such that 
$$\pi_1(G(x))\leq \pi_1(G(g)) \hbox{ for all }x\in A$$
and
$$\pi_1(G(y))\geq \pi_1(G(g)) \hbox{ for all }y\in B.$$
For $A,B\subset \R^2$ we write $G^A_B$ to denote the set of elements of $G(2,1)$ that divide $A$ and $B$.

If $A$ and $B$ are sets that can be divided by some $G\in G(2,1)$ and which have a common point $z$, then the angle between $A$ and $B$, $\psi^A_B$ is defined by 
$$\psi^A_B:=\min\{\theta:C(z,G+z,tan\theta) \supset A\cup B, G\in G^A_B\}.$$
\end{Definition}

We show that a set $\At$ does not spiral too tightly by showing that for each triangular cap, $T_{n,j}$, there is an appropriately large neighbourhood of $T_{n,j}$, $R_{n,j}$, in which $\At\cap R_{n,j}$ meets only $T_{n,j}$ and its direct neighbours.
\begin{Lemma}\label{lem7}
Let $\Theta \in \Psi$. Then 

(i) should two neighbouring trianglular caps, $T_{n,i}$ and $T_{n,i+1}$, be contained in another (necessarily earlier) triangular cap $T_{m,j(i)}$ ($m < n$) then $\psi^{T_n,i}_{T_{n,i+1}} \leq 2 \theta_m \leq 2 \theta_1$ and

(ii) the rectangle 
$$R_{n,i} := \pi_1 \left( O_{A_{n,i}}\left(\cup_{k:|i-k|\leq 1}T_{n,k}\right) \right) \times 
[-2\Hm{1}(A_{n,i}),2\Hm{1}(A_{n,i})]$$ satisfies
$$O_{A_{n,i}}^{-1}(R_{n,i}) \cap \At \subset \bigcup_{j:|i-j|\leq 1}T_{n,j}.$$
\end{Lemma}
\begin{proof}
For (i), let $T_{n,i}$ and $T_{n,i+1}$ be two neighbouring triangular caps with common point $z$. Then, by the 
construction of $\At$, $z = z_{n_1+1,2i_1}$ is the vertex of a triangular cap 
$T_{n_1,i_1}$ for some $m \leq n_1 <n$. (where in general $z_{n,i}:=T_{n,i-1}\cap T_{n,i}$).

Define
$$G_{n_1,i_1}^+(z):=\{\lambda(\pi_{G_{n_1,i_1}z}(z_{n_1+1,2i_1+1})-z):\lambda \geq 0\}+z \hbox{ and}$$
$$G_{n_1,i_1}^-(z):=\overline{G_{n_1,i_1}z\sim G_{n_1,i_1}^+(z)}$$
so that 
$$G_{n_1,i_1}^+(z)\cap G_{n_1,i_1}^-(z)=\{z\},$$
$$\psi^{G_{n_1,i_1}^-(z)}_{A_{n_1+1,2i_1}}=\theta_{n_1}/2, \hbox{ and }\psi^{A_{n_1+1,2i_1-1}}_{G_{n_1,i_1}^+(z)}=\theta_{n_1}/2.$$
We deduce that  $G_{n_1,i_1}$ divides $A_{n_1+1,2i_1}$ and $A_{n_1+1,2i_1-1}$ and that
\[\psi^{A_{n_1+1,2i_1}}_{A_{n_1+1,2i_1-1}} \leq \theta_{n_1}.\]
Since $T_{n_1+1, 2i_1-1}$ and $T_{n_1+1, 2i_1}$ are constructed on the interior of $T_{n_1, i_1}$ with base 
angle $\theta_{n_1+1}$, it follows that 
\[\psi^{T_{n_1+1,2i_1-1}}_{G_{n_1,i_1}^++z} \leq \theta_{n_1}+\theta_{n_1+1} \hbox{ and }
\psi_{T_{n_1+1,2i_1}}^{G_{n_1,i_1}^-+z} \leq \theta_{n_1} + \theta_{n_1+1}\]
and therefore that
\[\psi^{T_{n_1+1,2i_1-1}}_{T_{n_1+1, 2i_1}} \leq \theta_{n_1} + \theta_{n_1+1}.\]
Now, since $\{\theta_n\}$ is a non-increasing sequence, $\theta_{n_1+1}\leq \theta_{n_1} \leq \theta_{m} \leq \theta_1$ 
and thus 
$$\psi^{T_{n_1+1,2i_1-1}}_{T_{n_1+1, 2i_1}} \leq 2\p{m} \leq 2\p{1}.$$
We finally note that $T_{n,i} \subset T_{n_1,i_1}$ and $T_{n,i+1} \subset T_{n_1,i_1+1}$ so that 
$$\psi^{T_{n,i}}_{T_{n,i+1}} \leq 2\p{m} \leq 2\p{1}$$ proving $(i)$. 

To prove (ii) we first prove the claim that if $T_{n,i}$ and $T_{n,j}$ are triangular caps with $2 \leq |i-j| \leq 3$ then 
\[\pi_1\left(O_{A_{n,i}}\left(\bigcup_{k:|i-k|<2}T_{n,k}\right) \right) \cap 
\pi_1(O_{A_{n,i}}(T_{n,j})-\{z_{n,i-1},z_{n,i+2}\}) = \emptyset.\]

We prove the case for $j-i > 0$, the other case following symmetrically. From (i), $\psi^{T_{n,i}}_{T_{n,i+1}} \leq 2\p{1}$
and $\psi^{T_{n,i+1}}_{T_{n,i+2}} \leq 2\p{1}$. We deduce that
$\psi^{T_{n,i}}_{T_{n,i+2}-(z_{n,i+2} - z_{n,i+1})} \leq 4\p{1}.$
Moreover, since $\psi^{T_{n,i+2}}_{T_{n,i+3}} \leq 2\p{1},$
\[\psi^{T_{n,i}}_{T_{n,i+3}-(z_{n,i+3} - z_{n,i+1})} \leq \psi^{T_{n,i}}_{T_{n,i+1}} + \psi^{T_{n,i+1}}_{T_{n,i+2}} + \psi^{T_{n,i+2}}_{T_{n,i+3}} \leq 6\p{1}.\]
It thus follows that $\psi^{A_{n,i}}_{T_{n,i+3} - (z_{n,i+3} - z_{n,i+1})} \leq 6\p{1}$. 

Set $G_0\in G(1,2)$ to be the line minimising $\psi^{A_{n,i}}_{T_{n,i+3} - (z_{n,i+2} - z_{n,i+1})}$ and $G:=G_0+z_{n,i+1}$.
We then have 
$$C(z_{n,i+1},G,6\theta_1)\supset A_{n,i}\cup(T_{n,i+3}-(z_{n,i+3}-z_{n,i+1})),$$ 
and thus
$$O_{A_{n,i}}(C(z_{n,i+1},G,6\theta_1))\subset C((\Hm{1}(A_{n,i})/2,0),\R,12\theta_1).$$
For the appropriate selection of two possible $O_{A_{n,i}}$ it follows that 
$$O_{A_{n,i}}(T_{n,i+3}-(z_{n,i+3}-z_{n,i+1}))\subset C^+((\Hm{1}(A_{n,i})/2,0),\R,12\theta_1).$$
where for 
$x,\theta\in \R$, 
$$C^+(x,\R,\theta):=C(x,\R,\theta)\cap\{y:\pi_1(y)\geq x\}$$
and 
$$C^-(x,\R,\theta):=C(x,\R,\theta)\cap\{y:\pi_1(y)\leq x\}$$
We deduce, for $z_3:=z_{n,i+3}-z_{n,i+1}$, that
$$O_{A_{n,i}}(T_{n,i+3})\subset C^+((\Hm{1}(A_{n,i})/2,0)+z_3,\R+z_3 ,12\theta_1).$$
This being the worse of the two possible cases for $j$, namely $j=i+2$ or $j=1+3$, an identical procedure can be used to show that 
$$O_{A_{n,i}}(T_{n,i+2})\subset C^+((\Hm{1}(A_{n,i})/2,0)+z_2, \R+z_2,12\theta_1)$$
where $z_2:=z_{n,i+2}-z_{n,i+1}$.

Since $8\p{1} < 12\p{1} < \frac{\pi}{2}$, it follows that
\[\pi_1(O_{A_{n,i}}(T_{n,i+2} \cup T_{n,i+3})) \subset [\pi_1(O_{A_{n,i}}(z_{n,i+1})),\infty)\]
and
\begin{equation}\label{geqlem7}
\pi_1(O_{A_{n,i}}(T_{n,i+2} \cup T_{n,i+3})-\{z_{n,i+1},z_{n,i-2}\}) \subset (\pi_1(O_{A_{n,i}}(z_{n,i+1})),\infty).
\end{equation}
A similar argument yields 
$$O_{A_{n,i}}(T_{n,i+1}) \subset C^+((\Hm{1}(A_{n,i})/2,0),\R,4\theta_1)$$
and 
$$O_{A_{n,i}}(T_{n,i-1}) \subset C^-(-(\Hm{1}(A_{n,i})/2,0),\R,4\theta_1),$$
 so that since $4\p{1} < \pi/2-\p{1}$
\begin{eqnarray}
\max\{\pi_1(y):y \in O_{A_{n,i}}(T_{n,i+1})\} & = & \pi_1(O_{A_{n,i}}(z_{n,i+2})) \nonumber \\
& > & \pi_1(O_{A_{n,i}}(z_{n,i+1})) \nonumber \\
& = & \max\{\pi_1(y):y \in O_{A_{n,i}}(T_{n,i})\} \nonumber \\
& = & \pi_1(O_{A_{n,i}}(z_{n,i})) + \Hm{1}(A_{n,\cdot}) \nonumber \\
& \geq & \max\{\pi_1(y):y \in O_{A_{n,i}}(T_{n,i-1})\}. \nonumber 
\end{eqnarray}
Thus clearly
$\pi_{1}\left(\cup_{j:|i-j|<2}O_{A_{n,i}}(T_{n,j})\right) \subset (-\infty,\pi_1(O_{A_{n,i}}(z_{n,i+2}))],$
which, together with ($\ref{geqlem7}$) proves the claim.

We now prove (ii) by induction over $n$. Since there is $1$ triangular cap in $A_0$ and there are 2 triangular caps in $A_1$ the result is obvious for $n=0$ and $n=1$. For $A_2$ there are four triangular caps, meaning that there is 
something to prove. However, we note that for any chosen $i$ every trianglular cap is either in the triple around $i$ or has an index $j$ satisfying $2 \leq |i-j| \leq 3$. Since $\At$ is a subset of the four triangular caps, the required result follows directly from the above proved claim.

We now prove the inductive step. We suppose that the hypothesis holds for all triples 
$\{T_{p,i-1}, T_{p,i},T_{p,i+1}\}$ for a given $p \in \mathbb{N}$ and show that it holds for an arbitrary triple 
$\{T_{p+1,i-1}, T_{p+1,i},T_{p+1,i+1}\}$. 
We set 
\[\mathcal{T} := \cup \{T_{p+1,i-1}, T_{p+1,i},T_{p+1,i+1}\}.\]
Note first that 
\[\bigcup_{j:|i-j|<2}T_{p+1,j} \subset \bigcup_{j:|i_1-j|<2}T_{p,j},\]
where $i_1=j(p+1,i)$, so that the triple is in fact a subset of a triple in the $p$th construction level. This triple in the $p$th construction level, by construction, contains exactly $6$ trianglular caps in the ($p+1$)th construction level, namely 
$\{T_{p+1,j}\}_{j=2i_1-3}^{2i_1+2}$ with $T_{p+1,i} \in \{T_{p+1,2i_1-1},T_{p+1, 2i_1}\}$. We also have by the inductive 
hypothesis that 
\[\At \cap R_{p,i_1} \subset \bigcup_{j=2i_1-3}^{2i_1+2}T_{p+1,j}.\]
It follows that 
\[\At \cap R_{p+1,i} \cap R_{p,i_1} \subset \bigcup_{j=2i_1-3}^{2i_1+2}T_{p+1,j}.\]
Now, since $i \in \{2i_1-1, 2i_1\}$ we see that for all $j \in \{2i_1-3, ..., 2i_1+2\}$, either $|i-j| < 2$ 
or $2 \leq |i-j| \leq 3$. From the claim above it follows that for each $j$ such that 
$2 \leq |i-j| \leq 3$, $(T_{p+1, j} \sim \mathcal{T}) \cap R_{p+1,i} = \emptyset$. Thus
$\At \cap R_{p+1,i} \cap R_{p,i_1} \subset \mathcal{T}.$
The proof is thus complete in the case that $R_{p+1,i} \subset R_{p,i_1}$, as in this case 
\[\At \cap R_{p+1,i} = \At \cap R_{p+1,i} \cap R_{p,i_1} \subset \mathcal{T}.\]
It is clearly sufficient to show that $O_{A_{p,i_1}}(R_{p+1,i}) \subset O_{A_{p,i_1}}(R_{p,i_1})$. For $p\in \mathbb{N}$ and $j\in \{1,...,2^p\}$, define $\Hm{}_{p,j}:=\Hm{1}(A_{p,j})$. Without loss of generality we may assume that 
\begin{eqnarray}
O_{A_{p,i_1}}(T_{p+1,i}) & \subset & \triangle \left((0,0),\left(-\frac{\Hm{}_{p,j}}{2},0\right),\left(0,\frac{tan\theta_1 \Hm{}_{p,j}}{2}\right)\right) \nonumber \\
& \subset &\triangle \left((0,0),\left(-\frac{\Hm{}_{p,j}}{2},0\right),\left(0,\frac{\Hm{}_{p,j}}{100}\right)\right)\label{triincl}
\end{eqnarray}
where $\triangle(a,b,c)$ denotes the triangle in $\R^2$ with vertices $a,b$ and $c$. 
The other cases follow with symmetric arguments.

By the selection of $\Theta$, the construction of $\At$ and ($\ref{triincl}$), we see that
\begin{equation}\label{pi1rpi}
\pi_1\left(\bigcup_{j:|i-j|<2}O_{A_{p,i_1}}(T_{p+1,j})\right)\subset [-1.5,1]\Hm{}_{p+1,i}.
\end{equation}
Set $\eta:=2\cos(3\pi/8)\Hm{}_{p+1,i}$. Since $2\theta_{p+1}\leq 2\theta_1<\pi/8$, it follows from ($\ref{pi1rpi}$) that
$$\pi_1(O_{A_{p,i_1}}(R_{p+1,i}))\subset[-1.5-\eta, 1+\eta]\Hm{}_{p+1,i} \subset [-2.3,1.8]\Hm{}_{p+1,i}.$$
By ($\ref{triincl}$) and since $\pi_2(O_{A_{p,i_1}}(T_{p+1,i}))\leq 2sin\theta_1 \Hm{}_{p+1,i} \leq 0.15\Hm{}_{p+1,i}$ we also have
$$\pi_2(O_{A_{n,i_1}}(R_{p+1,i}))\subset [-2,2.15]\Hm{}_{p+1,i}$$
and therefore
\begin{equation}\label{rpiincl}
O_{A_{n,i_1}}(R_{p+1,i})\subset [-1.15,0.9]\Hm{}_{p,i_1}\times [-1,1.1]\Hm{}_{p,i_1}.
\end{equation}
Using again $\theta_p\leq \theta_1 \leq \pi/8$ we calculate $|\pi_1(O_{A_{p,i_1}}(T_{p,j}))|>0.9\Hm{}_{p,j}$ for $j=i_1\pm 1$, and therefore that
$$\pi_1(O_{A_{p,i_1}}(R_{p,i_1}))\supset [-1.4,1.4]\Hm{}_{p,i_1}.$$
Hence, by ($\ref{rpiincl}$)
$$O_{A_{p,i_1}}(R_{p,i_1})\supset [-1.4,1.4]\Hm{}_{p,i_1} \times [-2,2]\Hm{}_{p,i_1}\supset O_{A_{p,i_1}}(R_{p+1,i}),$$
completing the proof of (ii). 
\end{proof}
To appropriately apply the restrictions on spiraling we also require estimates on the lengths of the line segments $A_{n,i}^{\Theta}$.
\begin{Lemma}\label{lem5} 
Let $\Theta=\{\theta_n\}_{n=1}^{\infty}\in \Psi$ and $m\in \N$. Then
$$\Hm{1}(A_{m,i}^{\Theta})=2^{-m}\prod_{i=1}^m(cos\theta_i)^{-1}$$
and thus 
$$\Hm{1}(A_m^{\Theta})=\prod_{i=1}^m(cos\theta_i)^{-1}.$$
\end{Lemma}
\begin{proof}
By the construction of the $\theta_1$-triangular cap on $A_0$, $T_{0,1}$, we see that
$$\Hm{1}(A_{1,i})=\frac{\Hm{1}(A_0)}{2}(cos\theta_1)^{-1}=2^{-1}(cos\theta_1)^{-1}$$
as required.

Similarly, assuming
$$\Hm{1}(A_{m,i})=2^{-m}\prod_{i=1}^m(cos\theta_i)^{-1}$$
for some $m$ and $i$, we see by the construction of the $\theta_{m+1}$-triangular cap on $A_{m,i}$, $T_{m,i}$ that
$$\Hm{1}(A_{m+1,j})=\frac{\Hm{1}(A_{m,i})}{2}(cos\theta_{m+1})^{-1}=2^{-m-1}\prod_{i=1}^{m+1}(cos\theta_i)^{-1},$$
for $j\in \{2i-1,2i\}$, so that the result now follows by induction.
\end{proof}

With the above estimates on the measure of the $A_{n,i}$ and on the rate of spiraling, we can now present a general result describing when a subset of a set $\At$ satisfies a rather weak Reifenberg property. 
\begin{Lemma}\label{generalreif}
Let $\eta>0$ and $\Theta \in \Psi$. Suppose that for all $x\in A\subset \At$ there is an $r_x>0$ such that 
$$n_x:=\max\{1,\min\{n : B_{3r_x}(x)\cap E_{n+1}(\Theta)\not=\emptyset\}\}$$
satisfies 
$$5\theta_{n_x}<\psi_{\eta}:=tan^{-1}(\eta).$$
Then $A\in R(w,\rho,\eta;1)$.
\end{Lemma}
\begin{proof}
Let $x\in A$. We show that for each $y\in B_{r_x}(x)$ and $r\in (0,r_x]$ there is an $L_{y,r} \in G_y(2,1)$ such that 
$$\At\cap B_r(y)\subset L_{y,r}^{\eta r}.$$
Since $\theta_n$ is monotonically non-increasing
\begin{equation}\label{gr1}
\theta_n < 2\theta_{n-1}+3\theta_{n-2}<\psi_{\eta} \ \ \hbox{ for all } \ \ n\geq n_x+2.
\end{equation} 
Furthermore, for each $y\in B_{r_x}(x)\cap A$ and $r\in (0,r_x]$ there is, by Lemma $\ref{lem5}$ and the selection of $n_x$, an $n_1\geq n_x+2$ satisfying 
\begin{equation}\label{gr2}
\Hm{1}(A_{n_1,j})\in [r/2,r)
\end{equation}
for each $j\in \{1,...,2^{n_1}\}$. We now choose 
$$L_{y,r}:=G_{n_1,i(n_1,y)}+y.$$
By ($\ref{gr2}$) and Lemma $\ref{lem7}$ we see that 
\begin{equation}\label{gr3}
B_r(y)\cap T_{n_1,j}=\emptyset \ \ \hbox{ for } |j-i(n_1,y)|>2.
\end{equation}
From ($\ref{gr1}$) and ($\ref{gr2}$) and the selection of $n_1$ we infer that the vertical height of $T_{n_1,i(n_1,y)}$ is smaller than $r\eta$, thus
\begin{equation}\label{gr4}
T_{n_1,i(n_1,y)}\subset L_{y,r}^{\eta r}
\end{equation}
and, in particular,
\begin{equation}\label{gr5}
\{z_{n_1,i(n_1,y)},z_{n_1,i(n_1,y)+1}\}\subset C(y,L_{y,r},\psi_{\eta}).
\end{equation}
Moreover, from Lemma $\ref{lem7}$, ($\ref{gr1}$) and the selection of $n_1$ we deduce
\begin{equation}\label{gr6}
\psi^{T_{n_1,i}-(z_{n_1,i}-z_{n_1,i(n_1,y)+1})}_{T_{n_1,i(n_1,y)}}<2\phi_{n_x} + 2\phi_{n_1}<\psi_{\eta} 
\end{equation}
for $i\in \{i(n_1,y)+1,i(n_1,y)+2\}$ and
\begin{equation}\label{gr7}
\psi^{T_{n_1,i(n_1,y)}}_{T_{n_1,i+(z_{n_1,i(n_1,y)}-z_{n_1,i+1})}} <2\phi_{n_x} + 2\phi_{n_1}<\psi_{\eta}
\end{equation}
for $i\in \{i(n_1,y)-2,i(n_1,y)-1\}$. From ($\ref{gr5}$), ($\ref{gr6}$) and ($\ref{gr7}$) it follows that 
\begin{equation}\label{gr8}
T_{n_1,i}\subset C(y,L_{y,r},\psi_{\eta}) \ \ \hbox{ for } \ \ |i-i(n_1,y)|<2.
\end{equation}
Since
$$C(y,L_{y,r}, \psi_{\eta})\cap B_r(y) \subset L_{y,r}^{\eta r}$$
it follows from ($\ref{gr3}$), ($\ref{gr4}$) and ($\ref{gr8}$) that 
\begin{eqnarray}
\At\cap B_r(y) & \subset & \left(\bigcup_{|i-i(n_1,y)|}T_{n_1,i}\right)\cap B_r(y) 
 \subset  (T_{n_1,i(n_1,y)}\cup C(y,L_{y,r},\psi_{\eta}))\cap B_r(y) 
 \subset  L_{y,r}^{\eta r}. \nonumber
\end{eqnarray}
\end{proof}

For the purposes of our classification we need sets $\At$ satisfying differing Reifenberg and measure properties. Suitable candidates can be chosen via appropriate selection of $\Theta$.
\begin{Definition}\label{psic}
Define 
$$\Psi_c:=\{\Theta=\{\theta_n\}_{n=1}^{\infty}\in \Psi:\theta_n=\theta_1 \hbox{ for all } n\in \mathbb{N}\},$$
$$\Psi_{0}:=\{\Theta\in \Psi:\lim_{n\rightarrow \infty}\theta_n =0\},$$
$$\calA(\Psi_c,\theta):=\{\At:\Theta=\{\theta_n\}_{n\in \N}\in \Psi_c:\theta_1\leq\theta\},$$
$$\calA(\Psi_c):=\bigcup_{\theta\in (0,\pi/24)}\calA(\Psi_c,\theta), \hbox{ and } \calA(\Psi_0):=\{\At:\Theta \in \Psi_0\}.$$

\end{Definition}
\begin{Remark}
We show that the sets in $\calA(\Psi_0)$ and $\calA(\Psi_c)$ are the elements of $\calA(\Psi)$ that satisfy the appropriate Reifenberg properties.
\end{Remark}

\begin{Lemma}\label{lem10} 
\begin{enumerate}[(i)]
\item $\{A=\At\sim E(\Theta):\Theta \in \Psi_0\}\subset R(w,\rho,\delta;1)$.
\item 
There is a monotone increasing function $\eta:\R\rightarrow \R$ such that, for any $j\in \N$ and $\beta>0$,
$$\calA(\Psi_c,\eta(\beta))\subset R(w,\rho,\beta;1).$$
\end{enumerate}
\end{Lemma}
\begin{proof}
For (i): let $\beta>0$ and $\cl{A}\in\{A=\At\sim E(\Theta):\Theta \in \Psi_0\} $. Since $\Theta\in \Psi_0$ there is an $n_0\in \N$ such that $5\theta_n<\psi_{\beta}:=tan^{-1}(\beta)$ for each $n \geq n_0$.

Since $\cl{A}\cap E_{n_0+1}(\Theta) =\emptyset$ and $E_{n_0+1}(\Theta)$ is closed, $d(x,E_{n_0+1}(\Theta))>0$ for each $x\in \cl{A}$. For each $x\in \cl{A}$ define, 
$$r_x:=\frac{d(x,E_{n_0+1}(\Theta))}{3}.$$
By selecting $r_x$ in this manner the hypotheses of Lemma $\ref{generalreif}$ are satisfied and we deduce that 
$$\cl{A}\in R(w,\rho,\beta;1).$$ As $\beta>0$ was arbitrary (i) follows.

For (ii): set $\eta(\beta):=(tan^{-1}(\beta))/5.$

If $\Theta \in \Psi$ now satisfies $\theta_1<\eta$, then the hypotheses of Lemma $\ref{generalreif}$ are satisfied for $A:=\At$ by choosing, for any $x\in \At$, any $r_x>0$. The result follows.
\end{proof}

\section{Irregular Properties}
Having shown that $\calN$, $\Lambda$, and particular elements of $\calA(\Theta)$ satisfy the required linear approximation properties, it remains to show that they each have the necessary measure properties to be counter examples to the appropriate questions.

It is easily shown that $\calN$, respectively $\Lambda$, are not strongly, respectively weakly, locally $\Hm{1}$-finite. For $A\in \calA(\Psi_c)$, $dim_{\Hm{}}A >1$ follows from its properties as a self-similar set (being a flattened Koch curve) which we show shortly. That such a set is not of locally finite Hausdorff measure and not rectifiable follows directly. For $A\not\in \calA(\Psi_c)$, $A$ is not actually self-similar and showing that $A$ is neither locally $\Hm{1}$-finite nor $1$-rectifiable requires a bit more effort.

\begin{Lemma}\label{easybadprops}
$\calN$ does not have strongly locally finite $\Hm{1}$-measure. $\Lambda$ does not have either strongly or weakly locally finite $\Hm{1}$-measure. 
\end{Lemma}
\begin{proof}
For $\calN$, consider $x:=(0,0) \in \R^2$. We see that for each $\rho>0$ there are infinitely many lines of length $\rho/2$ contained in $B_{\rho}(x)$ so that $\Hm{1}(B_{\rho}(x) \cap \calN)=\infty$. As this is true for each $\rho>0$ the result follows.

We consider again $x:=(0,0)$, which lies in $\Lambda$. Now, for any chosen $\delta_0>0$ and $\rho>0$ we see that there are infinitely many lines of length greater than or equal to $\rho$ within $B_{\rho}(x) \cap L$ and therefore that $\Hm{1}(B_{\rho}(x) \cap \Lambda)=\infty$, completing the proof.
\end{proof}

Showing that $dim_{\Hm{}}A>1$ for $A\in \calA(\Psi_c)$ also suffices to show both that $A$ has neither weakly nor strongly locally finite $\Hm{1}$-measure and that $A$ is not $1$-rectifiable. That $dim_{\Hm{}} A>1$ for any $A\in \calA(\Psi_c)$ follows directly from standard theory once we have established that $\At$ satisfies Hutchinson's \cite{hutch} open set condition.

\begin{Definition}\label{similitudes}
Let $(X,d)$ be a metric space. A transformation $\varphi:X\rightarrow X$ is called a \emph{contraction} if there is a $0<r<1$, called the \emph{contraction factor} of $\varphi$, such that $d(\varphi(x_1),\varphi(x_2))\leq r d(x_1,x_2)$ for all $x_1, x_2\in X$. For a collection $\calS:=\{s_1,...,s_Q\}$ of contractions, $A\subset X$, and $n\in \N$, we write
$$\Phi^n(A):=\bigcup\{s_{i_n}\circ...\circ s_{i_1}(A):(i_1,...,i_n)\subset \{1,...,Q\}^n\}.$$
In this case, a set $K$, satisfying 
\begin{equation}\label{attractor}
K=\bigcup_{s\in \calS}s(K) = \lim_{n\rightarrow \infty}\Phi^n(A),
\end{equation}
is called an \emph{attractor} of $\calS$. Here, $A$ is an arbitrary compact subset of $X$ and the limit is taken with respect to the Hausdorff distance on $X$.

In the case that $d(\varphi(x_1),\varphi(x_2))= r d(x_1,x_2)$ for all $x_1,x_2\in X$, $\varphi$ is called a \emph{similitude}. The unique compact attractor of a family of similitudes is called a self-similar set.
\end{Definition}
\begin{Remark}
That the compact attractors of families of contractions and, in particular, similitudes exist, are unique, and are independent of the choice of $A$ in ($\ref{attractor}$) is a standard result. See, for e.g., Federer \cite{federer}.
\end{Remark}

\begin{Definition}\label{opesnset}
Let $\calS:=\{S_1, ..., S_Q\}$ be a collection of similitudes on $\R^2$. $\calS$ is then said to satisfy the open set condition if there exists a non-empty open set $O$ such that
\begin{enumerate}
\item $\bigcup_{i=1}^QS_i(O)\subset O$, and 
\item $S_i(O)\cap S_j(O)=\emptyset$ if $i\not= j$.
\end{enumerate}
\end{Definition}
The key result concerning the dimension of $A\in \calA(\Psi_c)$ is a proof that $A$ is the unique compact attractor of a family of similitudes satisfying the open set condition, as the dimension of any such set is then given by Hutchinson's Theorem. 
\begin{Theorem}\label{hutchthm}(Hutchinson)

Suppose $\calS:=\{s_i\}_{i=1}^m$ is a family of similitudes on $\R^n$ satisfying the open set condition with contraction factors $\{\lambda_i\}_{i=1}^m$ and that $K$ is the unique compact attractor for $\calS$. Then $dim_{\Hm{}}K$ is the unique $r_{\calS}\in \R$ satisfying
$$\sum_{i=1}^m\lambda_i^{r_{\calS}}=1.$$
\end{Theorem}
\begin{Remark}
The existence of the unique number $r_{\calS}$ was proven by Mandelbrot \cite{mandel}. Hutchinson's theorem can be found in many books on fractal geometry as well as in Hutchinson's original paper \cite{hutch}.
\end{Remark} 

\begin{Proposition}\label{prop3}
Let $\Theta \in \Psi_c$. 

Then there exist similitudes $S_1, S_2:\R^2\rightarrow \R^2$ with contraction factor $\lambda:=(2cos\theta_1)^{-1}$ satisfying the open set condition for which $\At$ is the compact attractor. 

In particular
$$dim_{\Hm{}}\At=-\frac{ln 2}{ln \lambda}.$$
\end{Proposition}
\begin{proof}
As $\Theta\in \Psi_c$ we can define $\theta \in \R$ so that $\theta=\theta_n$ for all $n\in \N$. Let 
\[M:=\left( \begin{array}{cc}
\cos((-1)^i\theta-\pi) & -\sin((-1)^i\theta-\pi) \\
\sin((-1)^i\theta-\pi) & \cos((-1)^i\theta-\pi)  \end{array} \right)\]
for $i\in \{1,2\}$ and define
\[ S_1:=M_1\cdot \lambda \left(\begin{array}{c} x \\ y\end{array} \right) + \left( \begin{array}{c} 1 \\ 0\end{array} \right) \hbox{ and }  S_2:=M_2\cdot \lambda \left( \begin{array}{c} x \\ y\end{array} \right) + \left( \begin{array}{c} 1/2 \\ h\end{array} \right)\]
where $h:=(tan\theta)/2$ is the vertical height of $T_{0,1}$.

By inspection, the transformations $s_1$ and $s_2$ are similitudes with contraction factor $\lambda$. Direct calculation shows that 
$$T_1=s_1(T_0)\cup s_2(T_0).$$
Suppose that there is a bijection, 
$$j_n:\{1,2\}^n\leftrightarrow \{1,...,2^n\},$$
such that, for each $I=(i_1,...,i_n)\in \{1,2\}^n$,
\begin{equation}\label{hypoth}
s_{i_n}\circ ... \circ s_{i_1}(T_0)=T_{n,j_n(I)},
\end{equation}
giving $\Phi^n(T_0)=T_n$, as is the case for $n=1$. 
By showing that the same supposition then holds for $n+1$ we deduce, by induction, that the supposition holds for all $n\in \N$.

Now let $(i_1,...,i_{n+1})\in \{1,2\}^{n+1}$. We see 
\begin{equation}\label{attract1}
S_{i_1}(T_0)=T_{1,i_1}\subset T_{0}
\end{equation}
is a triangular cap on $A_{1,i_1}$, a shorter side of the isosceles triangle $T_0$. For $I_1:=(i_2,...,i_{n+1})\in \{1,2\}^n$ , the induction hypothesis gives
$$S_{i_{n+1}}\circ ...\circ s_{i_2}(T_0)=T_{n,j_n(I_1)}.$$
Since the transformations $s_1$ and $s_2$ are similitudes it follows that 
$$s_{i_{n+1}}\circ ... \circ s_{i_1}(T_0)$$
is a $\theta$-triangular cap on a shorter side of $T_{n,j_n(I_1)}.$ That is 
$$s_{i_{n+1}}\circ ... \circ s_{i_1}(T_0)\in \{T_{n+1,i}\}_{i=1}^{2^{n+1}} \hbox{ and } s_{i_{n+1}}\circ ... \circ s_{i_1}(T_0)\subset T_{n,j_n(I_1)}.$$
For $(i_1,...,i_{n+1})\not= (k_1,...,k_{n+1})\in \{1,2\}^{n+1}$ either
\begin{enumerate}[(i)]
\item $k_1\not=i_1$ or
\item $I_2:=(k_2,...,k_{n+1})\not=I_1.$
\end{enumerate}
In case (i) $S_{i_1}(T_0)=T_{1,i_1}\not=T_{1,k_1}=S_{k_1}(T_0)$ and thus 
$$s_{i_{n+1}}\circ ... \circ s_{i_1}(T_0)\not= s_{k_{n+1}}\circ ... \circ s_{k_1}(T_0).$$
In case (ii) 
$$s_{i_{n+1}}\circ ... \circ s_{i_1}(T_0)\subset T_{n,j_n(I_1)} \not= T_{n,j_n(I_2)}\supset s_{k_{n+1}}\circ ... \circ s_{k_1}(T_0).$$
It follows that $\{s_{i_{n+1}}\circ ... \circ s_{i_1}(T_0):(i_1,...,i_{n+1})\in \{1,2\}^{n+1}\}$ is a set of $2^{n+1}$  different elements of $\{T_{n+1,i}\}_{i=1}^{2^{n+1}}$. We deduce that the supposition, ($\ref{hypoth}$), also holds for $n+1$. We deduce that 
$$\Phi^n(T_0)=T_n$$
for all $n\in \N$,

Since $d_{\Hm{}}(T_n,\At)\leq d(T_{n,i})=\Hm{1}(A_{n,i}),$ and $\lim_{n\rightarrow \infty}\Hm{1}(A_{n,i})=0$, it follows that $\At$ is the attractor of $\{s_1,s_2\}$. 

Setting now $s=-ln 2 (ln \lambda)^{-1}$ we observe that 
$$\sum_{i=1}^2\lambda^s=1$$
so that the result now follows from Theorem $\ref{hutchthm}$.
\end{proof}
From the above Lemma we deduce the necessary measure properties for sets in $\calA(\Psi_c)$.
\begin{Lemma}\label{lem12}
For each $\At \in \calA(\Psi_c)$, $dim_{\Hm{}}\At>1$, $\At$ has neither strongly nor weakly locally $\Hm{1}$-finite measure, and is not countably $1$-rectifiable.
\end{Lemma}
\begin{proof}
From Proposition $\ref{prop3}$, 
$$dim_{\Hm{}}\At = -\frac{ln 2}{ln (2cos\theta_1)^{-1}} >1.$$
Since all compact subsets of $\R^2$ weakly locally $\Hm{1}$-finite, strongly locally $\Hm{1}$-finite, or $1$-rectifiable are of Hausdorff dimension 1, we deduce that $\At$ can possess none of these properties.
\end{proof}
\begin{Remark}\label{kochrectrem}
It is actually also true that each $\At\in \calA(\Theta)$ is purely $1$-unrectifiable. As in this case $\At$ is a Koch curve this is not a new result, however, the result also follows from our analysis of sets $\At \in \calA(\Psi)\cap R(w,\rho,\delta;1)$ below. See Corollary $\ref{psicunrect}$.
\end{Remark}

Showing that elements of $\calA(\Psi)\cap R(w,\rho,\delta;1)$ satisfy the necessary measure properties is somewhat more delicate as, by Lemma $\ref{lem2}$, any such set has Hausdorff dimension 1. A simple summary of measure properties, as in Lemma $\ref{lem12}$, via Proposition $\ref{prop3}$, is therefore not possible. A suitable example of a set in $R(w,\rho, \delta; 1)$ with the appropriate measure properties is a subset of an element of $\calA(\Psi_0)$. That $\Theta \in \Psi_0$ is, however, also not a sufficient condition. To find the conditions on $\Theta$ necessary to ensure the desired measure properties, we need first to consider the representation of a set $\At$ by a function.

\begin{Definition}\label{dyadic}
We define the dyadic points in $[0,1]$ by
$$D_n:=\{d_{n,j}:=j2^{-n}:j\in\{0,1,...,2^n\} \}\hbox{ and } D:= \bigcup_{n\in \N} D_n.$$
\end{Definition}
\begin{Definition}\label{curlyf} 
Let $\Theta \in \Psi$. We define $F_n:[0,1]\rightarrow A_n^{\Theta}$ to be the Lipschitz functions satisfying the conditions
\begin{enumerate}
\item $F_n^{\Theta}|_{[d_{n,j},d_{n,j+1}]}$ is linear, 
\item $F_n^{\Theta}(d_{n,j})=E_{n,j}^{\Theta}$ for $j\in \{0,1,...,2^n\}$, and 
\item $F_n^{\Theta}([d_{n,j},d_{n,j+1}])=A_{n,j+1}^{\Theta}$ for $j\in \{0,...,2^n-1\}$. 
\end{enumerate}
Define $\F1_{\Theta}:A_{0,1}\rightarrow \At$ by 
\[\F1_{\Theta}(x):=\lim_{n\rightarrow \infty}F_n^{\Theta}(x).\]
\end{Definition}

\begin{Remark}
(1) In the case that the $\Theta$ being refered to is clear, the sub and super script $\Theta$ will be omitted.

(2) For all $y\in [0,1]$ there is a sequence $j(n,y)$ such that 
$$y\in [d_{n,j(n,y)},d_{n,j(n,y)}] \hbox{ and } F_n(y)\in T_{n.j(n,y)}\subset T_{n-1.j(n-1,y)}.$$
Since $d(T_{n,j})\rightarrow 0 $ as $n\rightarrow \infty$, it follows that $\F1$ is well defined.

(3) It is easy to check that $\F1(D)=E$.

(4) By the definition and Lemma $\ref{lem5}$ it is clear, for each $n\in \N$, that $F_n$ is a Lipschitz function with
$$Lip F_n=\prod_{i=1}^n(cos\theta_i)^{-1}.$$
\end{Remark}

\begin{Proposition}\label{F1props}
For any $\Theta \in \Psi$, $\F1_{\Theta}$ is a bicontinuous bijection. 
\end{Proposition}
\begin{proof}
For $y,z \in A_{0,1}$, $y\not= z$ we can find $n,j,k \in \N$ in order that $y\in [d_{n,j-1},d_{n,j}]$ and $z\in [d_{n,k-1},d_{n,k}]$ with $|j-k|\geq 2$. It follows that  $\F1(y)\in T_{n,j}$ and $\F1(z)\in T_{n,k}$. Since $T_{n,k}\cap T_{n,j}=\emptyset$, $\F1(y)\not= \F1(z)$ and hence $\F1$ is injective.

Suppose $z\in \At$, then $z\in T_{n,i(n,z)}$ for each $n\in \N$. Choose 
$$x_z\in \bigcap_{n\in \N}[d_{n,i(n,z)},d_{n,i(n,z)+1}].$$ 
Then
$$\lim_{n\rightarrow \infty}|F_n(x)-z|\leq \lim_{n\rightarrow \infty}d(T_{n,i(n,z)}) = \lim_{n\rightarrow \infty}2^{-n}\prod_{i=1}^n(cos\theta_i)^{-1}=0,$$
hence $\F1(x)=z$ and $\F1$ is surjective.

Let $\eta>0$, then there is an $n\in \N$ such that $diam(T_{n,j})\in [\eta/4,\eta/2).$ Let $\delta = 2^{-n-2}$ and $x,y \in A_{0,1}$ satisfy $|y-x|<\delta.$ It follows that there exists $k\in \{0,1,...,2^n\}$ such that $x,y \in [d_{n,k-1},d_{n,k+1}]$, and therefore that $F_m(x),F_m(y) \in T_{n,k}\cup T_{n,k+1}$ for all $m \geq n$.

Let $m\in \N$, if $m \leq n$
\[|F_m(x)-F_m(y)|<Lip F_m \delta < (1+m16\e^2)^{1/2}2^{-n}<diam(T_{n,j})<\eta.\]
If $m > n$, $|F_m(x)-F_m(y)|<diam(T_{n,k})+diam (T_{n,k+1}) < \eta$. $\{F_n\}_{n\in \N}$ is therefore equicontinuous. 

$\{F_n\}_{n\in \N}$ is both equicontinuous and bounded, it therefore follows from the Arzela-Ascoli Theorem that $\F1$ is continuous. In turn, since $[0,1]$ is compact and $\F1$ is a continuous bijection, we deduce that $\F1^{-1}$ is continuous.
\end{proof}
$\F1_{\Theta}$ assists greatly in proving measure results concerning $\At$. We use $\F1_{\Theta}$ firstly to show the infinite measure of appropriately selected $\At$. For $\Theta:=\{\theta_i\}_{i\in \N}\in \Psi_0$, with sufficiently rapidly decreasing $\theta_i$, $\lim_{n\rightarrow \infty}\prod_{i=1}^n(cos\theta_i)^{-1}<\infty$, so that by Lemma $\ref{lem5}$ and simple calculations, as in the above Proposition, it follows that $\At$ is a Lipschitz curve. As a Lipschitz curve has very regular measure theoretic properties, such sets are not the sets sought. This motivates the next definition of a new subset of $\calA(\Psi)$ in which our counter example will be found.

\begin{Definition}\label{psiinfty}
Define
$$\Psi_{\infty}:=\left\{\Theta \in \Psi:\lim_{n\rightarrow \infty}\prod_{i=1}^n(cos\theta_n)^{-1}=\infty\right\}$$
and 
$$\calA(\Psi_{\infty}):=\{\At:\Theta \in \Psi_{\infty}\}.$$
\end{Definition}

\begin{Proposition}\label{HF1infinite}
Let $\Theta \in \Psi$ and $K\subset [0,1]$ satisfy $\Hm{1}(K)>0$. Then
$$\Hm{1}(\F1_{\Theta}(K))\geq \frac{\Hm{1}(K)}{8}.$$
Furthermore, if $\Theta \in \Psi_{\infty}$, $\Hm{1}(\F1_{\Theta}(K))=\infty$.
\end{Proposition}
\begin{proof}
If $\Theta \not\in \Psi_{\infty}$ let $M:=\Hm{1}(K).$ Otherwise, choose $M>0$ arbitrarily.

Since $LipF_n\geq 1$ for each $n\in \N$ and since, by Lemma $\ref{lem5}$, 
$$\lim_{n\rightarrow \infty}LipF_n=\infty$$
 whenever $\Theta \in \Psi_{\infty}$, there is an $n_0\in \N$ such that $\Hm{1}(F_{n_0}(K))>M.$

Note also, by Lemma $\ref{lem7}$, that we can define $n:(0,1)\rightarrow \N$ so that, for any set $B\subset \R^2$ with $d(B)<\delta$,
$$\Hm{1}(A_{n(\delta),i})\in (d(B),2d(B))$$
and 
$$|\{i\in \{1,...,2^{n(\delta)}\}:T_{n(\delta),i} \cap B \not= \emptyset \}|\leq 2.$$
Now let $\delta_0>0$ be chosen so that $n(\delta_0)>n_0$ and so that for each $0<\delta \leq \delta_0$
$$\Hm{1}_{\delta}(F_{n_0}(K))>\frac{M}{2}.$$
Let $\calB:=\{B_i\}_{i=1}^{\infty}$ be a $\delta$-cover of $\F1(K)$ for some $\delta<\delta_0/4$. 

Let $B\in \calB$ and note that $n_1:=n(d(B))>n_0.$ There are, therefore, $j_B, j_{B+1}\in \{1,...,2^{n_1}\}$ such that 
$$B\cap T_{n_1,j}=\emptyset \hbox{ for } j\not\in \{j_B, j_{B+1}\}.$$
It follows that $\F1^{-1}(B)\subset [(j_B-1)2^{-n_1}, (j_B+1)2^{-n_1}]$ with 
$$2^{1-n_1}=2\Hm{1}(A_{n_1,j_B})\prod_{i=1}^{n_1}cos\theta_i < 4d(B)\prod_{i=1}^{n_1}cos\theta_i.$$
We deduce, since $n_1>n_0$, that
\begin{eqnarray}
d(F_{n_0}(\F1^{-1}(B)))  <  4d(B)\prod_{i=1}^{n_1}cos\theta_i\prod_{i=1}^{n_0}(cos\theta_i)^{-1} 
< 4d(B) < \delta_0. 
\end{eqnarray}
Since $\{\F1^{-1}(B)\}_{B\in \calB}$ is a cover of $\F1^{-1}(K)$ we infer that $\{F_{n_0}(\F1^{-1}(B))\}_{B\in \calB}$ is a $\delta_0$-cover of $F_{n_0}(K)$ and thus that 
\begin{equation}\label{mon8}
\sum_{i=1}^{\infty}d(B)> \frac{1}{4}\sum_{i=1}^{\infty}d(F_{n_0}(\F1^{-1}(B)))>\frac{M}{8}.
\end{equation}
As this is true for each $\delta < \delta_0$ the first claim holds. In the case that $\Theta\in \Psi_{\infty}$, ($\ref{mon8}$) holds for any $M>0$ and the result follows.
\end{proof}
The above result suffices to show the existence of sets in $R(w,\rho, \delta; 1)$ without locally finite measure. 

In showing that $\calA(\Psi_0)\subset R(w,\rho,\delta;1)$, Lemma $\ref{lem7}$ showed that sets in $\calA(\Psi_0)$ do not spiral too tightly. In order to show non-rectifiability results, we need to show that sets $\At$ can spiral quickly enough. After establishing how we define and control the rotation of the approximating triangles and therefore the spiraling of sets in $\calA(\Psi)$, we show in Lemma $\ref{RN0andc}$ that, for appropriately selected $\Theta$, $\At$ does spiral appropriately. 

\begin{Definition}\label{thetani}
For $\theta\in (-\pi,\pi)$, let $R_{\theta}:\R^2\rightarrow \R^2$ be the rotation in the positive (that is, anticlockwise) direction by an angle of $\theta$. 

Let $\Theta \in \Psi$, we define $\vartheta_{n,i}^{\Theta}\in \{\theta_n,-\theta_n\}$ to be the angle satisfying
$$R_{\vt_{n,i}}(G_{n,i}^{\Theta})=G_{n-1,j(n,i)}^{\Theta}$$
As per usual, the superscipt $\Theta$ will be dropped in the case that the $\Theta$ being referred to is clear. 
\end{Definition}

\begin{Lemma}\label{vtniformula}
Let $\Theta \in \Psi$, then 
\begin{equation}\label{thetasign}
\vt_{n,i}=
\begin{cases}
\theta_n \ \ 2\not| (n+i) \cr
-\theta_n \ \ 2|(n+i), 
\end{cases}
\end{equation}
where, here, $a|b$ denotes that $a$ divides $b$.
\end{Lemma}
\begin{proof}
We first consider $\Theta \in \Psi_c$. By construction 
\begin{equation}\label{vtnif1}
T_{n,i}\cap T_{n,j}\not=\emptyset
\end{equation}
 if and only if $|i-j|\leq 1$. Furthermore, by definition, the similitudes $S_1$ and $S_2$ satisfy
 $$S_1((1,0))=(0,0), \ \ S_1((0,0))=\left(\frac{1}{2}, \frac{tan \theta}{2}\right),$$
 $$S_2((1,0))=\left(\frac{1}{2},\frac{tan \theta}{2}\right), \ \ \hbox{ and } S_2((0,0))=(1,0).$$
 Since, for each $n \in \mathbb{N}$, 
$$(0,0)\in T_{n,1}, \left(\frac{1}{2},\frac{tan \theta}{2}\right)\in T_{n,2^{n-1}}\cap T_{n,2^{n-1}+1}, \ \ \hbox{ and } (1,0)\in T_{n,2^n}$$
we see that 
$$S_1(T_{n,1})=T_{n+1,2^n}, S_1(T_{n,2^n})=T_{n+1,1},$$ 
$$S_2(T_{n,1})=T_{n+1,2^{n+1}}, \ \ \hbox{ and } \ \ S_2(T_{n,2^n})=T_{n+1,2^{n}+1}.$$
Combining with ($\ref{vtnif1}$) we deduce, for $n\in \mathbb{N}$ and $j\in \{1,...,2^n\}$ that 
\begin{equation}\label{vtnif2}
S_1(T_{n,j})=T_{n+1,2^n-j+1} \ \ \hbox{ and } S_2(T_{n,j})=T_{n+1,2^{n+1}-j+1}.
\end{equation}
Now, by construction of the isosceles triangle $T_{0,1}$ on $A_{0,1}$, we see that 
\begin{equation}\label{vtnif3}
\vt_{1,1}=-\theta_1 \ \ \hbox{ and }\vt_{1,2}=\theta_1.
\end{equation}
Further, since $S_1$ and $S_2$ are similitudes, and in particular, conformal mappings, if $S_k(T_{n,i})=T_{n+1,j}$ then $\vt_{n+1,j}=\vt_{n,i}$ for $k\in \{1,2\}, n\in \mathbb{N}$, $i\in \{1,...,2^n\}$ and $j\in \{1,...,2^{n+1}\}$.
By ($\ref{vtnif2}$) it follows that 
\begin{equation}\label{vtnif4}
\vt_{n+1,2^n-j+1}=\vt_{n,j}=\vt_{n+1,2^{n+1}-j+1}.
\end{equation}
By ($\ref{vtnif3}$), using ($\ref{vtnif4}$) inductively, and that $\theta_n=\theta_1=:\theta$ for each $n\in \mathbb{N}$, ($\ref{thetasign}$) follows.

For the more general case, $\Theta\in \Psi$. Since, in the construction of $\At$, the size but not the sign of $\vt_{n,i}$ varies when we allow $\theta_n$ to vary over $n$, ($\ref{thetasign}$) continues to hold.
\end{proof}

Lemma $\ref{vtniformula}$ describes the rotational change that occurs at each level of approximating sets. Most importantly, it describes which direction one triangular cap rotates with respect to the triangular cap in which it is constructed, namely, clockwise or anti-clockwise. This binary representation of rotation can be combined with a similar system, giving the location of a point $x\in \At$, to describe how often and how far the approximating sets rotate centered on a given point in $\At$.

\begin{Definition}\label{base2}
For $x\in [0,1]$ we write 
$$x=x_0.x_1x_2x_3...$$
to denote the base 2 representation of $x$. We use the convention that the sequence $\{x_i\}_{i=1}^{\infty}$ is the unique sequence to have infinitely many terms equal to $0$ such that 
$$x=\sum_{i=0}^{\infty}x_i2^{-i}.$$
\end{Definition}

\begin{Lemma}\label{i-1lxli}
Let $\Theta \in \Psi$, then for any $x\in [0,1]$, $\calF_{\Theta}(x)\in T_{n,i}$ if and only if 
$$i-1\leq \sum_{j=0}^{\infty}x_j2^{n-j}\leq i.$$
\end{Lemma}
\begin{proof}
By Definition $\ref{curlyf}$ $\calF_{\Theta}^{-1}(T_{n,i})=[(i-1)2^{-n},i2^{-n}]$, therefore $\calF(x)\in T_{n,i}$ if and only if 
$$(i-1)2^{-n}\leq x\leq i2^{-n}.$$ 
The result now follows with Definition $\ref{base2}$.
\end{proof}
\begin{Corollary}\label{absthetasum}
Let $\Theta \in \Psi$ and $x\in [0,1]$, then, if
$$x=x_0.x_1...x_{m_1-1}1010...10x_{m_2+1}x_{m_2+2}...$$
$$\left|\sum_{i=m_1}^{m_2}\vt_{n,i(n,\calF_{\Theta}(x))}\right|=\sum_{i=m_1}^{m_2}\theta_i.$$
\end{Corollary}
\begin{proof}
For $n\in \N$, we infer from Lemma $\ref{i-1lxli}$ that $i(n,\calF_{\Theta}(x))$ is even if and only if $x_n=1$. By Lemma $\ref{vtniformula}$ it follows either that $\vt_{n,i(n,\calF_{\Theta}(x))}=\theta_n$ for $m_1\leq n\leq m_2$, or that $\vt_{n,i(n,\calF_{\Theta}(x))}=-\theta_n$ for $m_1\leq n\leq m_2$. In either case, we deduce, as required, that
$$\left|\sum_{i=m_1}^{m_2}\vt_{n,i(n,\calF_{\Theta}(x))}\right|=\sum_{i=m_1}^{m_2}\theta_i.$$
\end{proof}
\begin{Definition}\label{RMAN}
For $N,M,M_0\in \mathbb{N}$, $a_1...a_M$ a string of digits in base 2, and $N\geq M+M_0$, define
$$R(M_0,a_1...a_m, N):=\{x\in [0,1]:\exists M_0\leq M_1\leq N-M:x=x_0.x_1...x_{M_1}a_1...a_Mx_{M_1+M+1}...\}.$$
\end{Definition}
\begin{Lemma}\label{RN0andc}
For any $M, M_0\in \mathbb{N}$, any $0<c<1$, and any base 2 string of digits $\alpha:=a_1...a_M$ there is an $N_0=N_0(\alpha,M_0,c)\in\N$ such that 
$$\Hm{1}(R(M_0,\alpha,N))\geq c$$
for all $N\geq N_0$.
\end{Lemma}
\begin{proof}
By the Normal Number Theorem, see, for example, Borel \cite{borel} or Niven \cite{niven}, there is a set $S\subset [0,1]$ with $\Hm{1}(S)=1$ such that the sequence $a_1...a_M$ occurs infinitely often in the base 2 expansion of each element $x\in S$.
In particular, for each $x\in S$
$$x=x_0.x_1...x_{M_0}x_{M_0+1}...x_{Q-1}a_1...a_Mx_{O+m+1}...$$
for infinitely many and, in particular, at least one $Q\in \mathbb{N}$. It follows that 
$$S\subset \bigcup_{N=M_0+M}^{\infty}R(M_0,\alpha,N)$$
and hence
$$\lim_{N\rightarrow \infty} \Hm{1}(R(M_0,\alpha,N))=\Hm{1}(S)=1.$$
The result now follows.
\end{proof}
The fact that arbitrary sequences can almost always be found motivates the following definition of a subset of $\Psi$ which allows the transfer of sequences of digits to magnitude of rotation. Since a continual rotation, loosely speaking, prevents the existence of tangent spaces, we can then exploit continual rotation to prove that certain sets in $\calA(\Psi_{\infty})$ are not rectifiable.

\begin{Definition}\label{psir}
For $\Theta\in \Psi$ define
$$R(\Theta):=\left\{x\in [0,1]:\forall n\in \mathbb{N}, \exists m_1,m_2\geq n:\sum_{j=m_1}^{m_2}\vt_{n,i(n,\calF_{\Theta}(x))}>2\pi\right\}$$
and, for $0<r\leq 1$, define
$$\Psi_R^r:=\{\Theta\in \Psi:\Hm{1}(R(\Theta))\geq r\}.$$
\end{Definition}
\begin{Remark}
$R(\Theta)$ is in general neither open nor closed. $R(\Theta)$ is, however, always a Borel set and therefore measurable, a property that is necessary in our analysis of rectifiability below.
\end{Remark}

\begin{Proposition}\label{rmeas}
Let $\Theta \in \Psi$, then $R(\Theta)$ is a Borel set and therefore $\Hm{1}$-measurable.
\end{Proposition}
\begin{proof}
For $k> n$ let
$$R(\Theta,n,k):=\left\{x\in [0,1]:\exists n\leq m_1,m_2\leq k:\sum_{j=m_1}^{m_2}\vt_{n,i(n,\calF_{\Theta}(x))}>2\pi\right\}.$$
We see that 
\begin{equation}\label{runion}
R(\Theta)= \bigcap_{n\in \N}\bigcup_{k>n}R(\Theta,n,k).
\end{equation}
By Definition $\ref{inx}$ and Lemma $\ref{i-1lxli}$,  $ R(\Theta,n,k)\cap [j2^{-k},(j+1)2^{-k})\not=\emptyset$ only if 
$[j2^{-k},(j+1)2^{-k}) \subset R(\Theta,n,k)$. It follows that for each $n\in \N$ and $k>n$
$$R(\Theta,n,k) = \bigcup_{j\in I}[j2^{-k},(j+1)2^{-k})$$
for some $I \subset \{0,1,...,2^k-1\}$. We deduce, for each $n\in \N$ and $k>n$, that $R(\Theta,n,k)$ is a Borel set and therefore, by ($\ref{runion}$), that $R(\Theta)$ is a Borel, and thus $\Hm{1}$-measurable, set. 
\end{proof}
In proving our non-rectifiability results we use some standard characterisations of rectifiability which we now recall for reference. Proofs can be found, for example, in \cite{mattila} and \cite{simon1}. 

\begin{Theorem}\label{tspace}
For a $\Hm{1}$-measurable set $A\subset\R^2$ the following conditions are equivalent:
\begin{enumerate}[(i)]
\item $A$ is $1$-rectifiable,
\item For $\Hm{1}$-almost all $x\in A$ there is a unique approximate tangent $1$-plane for $A$ at $x$. That is ,there is a unique $V\in G(2,1)$ such that, for all $0<s<1$,
$$\lim_{r\rightarrow 0}r^{-1}\Hm{1}((A\cap B_r(x))\sim C(x,V,s))=0.$$
\item There is a positive locally $\Hm{1}$-integrable function, $\theta$, on $A$  with respect to which a unique approximate tangent space, $T_xA$, exists for $\Hm{1}$-almost all $x\in A$. That is, there exists a unique $P\in G(2,1)$ such that 
$$\lim_{\l\rightarrow 0}\int_{\eta_{x,\l}A}\psi d\Hm{1}=\theta(x)\int_P\psi d\Hm{1}$$
for all $\psi \in C_C^0(\R^2)$, where $\eta_{x,\l}B:=\l^{-1}(B-x)$ for all $B\subset \R^2$.
\end{enumerate}

\end{Theorem}

\begin{Lemma}\label{unrect}
Let $\Theta\in \Psi$ and $\Hm{1}(R(\Theta))>0$. Then $\calF_{\Theta}(R(\Theta))$ is purely unrectifiable.
\end{Lemma}
\begin{proof}
By Proposition $\ref{HF1infinite}$, $\Hm{1}(\Ft(R(\Theta)))>0$. Now, suppose that $E\subset \Ft(R(\Theta))$ is rectifiable and that $\Hm{1}(E)>0$. As $E$ is rectifiable, we can write
$$E\subset M_0\cup \bigcup_{i=1}^{\infty}f_i(\R)$$
for Lipschitz functions $f_1$ and $\Hm{1}(M_0)=0$. We can therefore find $j\in \N$ and $M>0$ such that 
$$\Hm{1}(E\cap f_j([-M,M])>0.$$ 
Setting $E\subset F:= \overline{F_j([-M,M])}\cap \calF(R(\Theta))$ we see, by Proposition $\ref{rmeas}$, that $F$ is a $1$-rectifiable measurable set satisfying
$$0<\Hm{1}(F)<2MLipf_j <\infty \hbox{ and } F_1\subset \calF(R(\Theta)).$$
Now, by Theorem $\ref{tspace}$, for $\Hm{1}$-almost all $x\in F$
\begin{enumerate}[(i)]
\item There is a unique approximate tangent space $T_xF$ of $F$ at $x$, and 
\item There is a unique approximate tangent $1$-plane for $F$ at $x$.
\end{enumerate}
Since $\Hm{1}(F)>0$ we can choose $x\in F$ such that (i) and (ii) hold. Let $P_x$ and $V_x$ be the approximate tangent space and approximate tangent $1$-planes for $F$ at $x$ respectively.

From (i) it follows that 
\begin{equation}\label{unrect1}
\Hm{1}((F\cap B_{\rho}(x))\sim B_{\rho/2}(x))>\frac{\theta\rho}{2}
\end{equation}
for all sufficiently small $\rho$, say $0<\rho\leq \rho_0<1$. From (ii) we deduce that for all sufficiently small radii, say $0<\rho\leq \rho_1\leq \rho_0$,
\begin{equation}\label{unrect2}
\Hm{1}((F\cap B_{\rho}(x))\sim C(x,V_x,tan(\pi/8))<\frac{\theta\rho}{2}.
\end{equation}
Let $n_0\in \N$ be such that $\Hm{1}(A_{n_0,i(n_0,x)})<\rho_1$. Since $\Ft^{-1}(x)\in R(\Theta)$ there are $m_1, m_2>n_0$ such that 
$$\sum_{n=m_1}^{m_2}\vt_{n,i(n,x)}>2\pi.$$
Since $\theta<\pi/24$ for each $n\in\N$ there is an $m_0\geq n_0$ such that $V_x$ and $G_x:=G_{m_0,i(m_0,x)}$ meet at an angle of $\vt_m\in(\pi/2-\pi/24,\pi/2+\pi/24)$. (In our present case it is irrelevant which of the angles between $V_x$ and $G_x$ is used.) For $r\in (\Hm{1}(A_{m_0,i(m_0,x)})/2,\Hm{1}(A_{m_0,i(m_0,x)}))$ we note that $B_r(x)\subset R_{m_0,i(m_0,x)}$ and thus, with Lemma $\ref{lem7}$ parts (i) and (ii), we can easily calculate that
\begin{equation}\label{unrect3}
(F\cap B_r(x))\sim B_{r/2}(x))\subset (\At\cap B_r(x))\sim B_{r/2}(x) \subset C(x,G_x+x,\pi/4).
\end{equation}
Noting that $r<\Hm{1}(A_{m_0,i(m_0,x)}) < \rho_1$ and $C(x,V_x,tan(\pi/8))\cap C(x,G_x+x,\pi/4)=\{x\}$ we deduce from ($\ref{unrect1}$) and ($\ref{unrect3}$) that
\begin{eqnarray}
\Hm{1}((F\cap B_{\rho_2}(x)) \sim C(x,V_x,tan(\pi/8))) 
 \geq  \Hm{1}(((F\cap B_{\rho_2}(x))\sim B_{\rho_2/2}(x))\sim\{x\})
> \frac{\theta\rho_2}{2}, \nonumber
\end{eqnarray}
which contradicts ($\ref{unrect2}$). We deduce that $F$ is not rectifiable, contradicting the selection of $F$.
\end{proof}
Lemma $\ref{unrect}$ completes the technical analysis of the sets in $\calA(\Psi)$. In order to give proofs of our main theorems, though, we need show that appropriate specific examples can be selected. That is, we need to show that there are $\Theta\in \Psi$ that are simultaneously elements of all necessary subfamilies of $\Psi$. In particular, we need to show that $\Psi_c\subset \Psi_{\infty}\cap \Psi_R^1$ and that $\Psi_0\cap \Psi_{\infty}\cap \Psi_R^1\not=\emptyset$.

\begin{Lemma}\label{tripleintersection}
$$\Psi_0\cap \Psi_{\infty}\cap \Psi_R^1\not=\emptyset.$$
\end{Lemma}
\begin{proof}
We note that $\{\pi/24n\}_{n\in \N}\in \Psi_0$ so that $\Psi_0\not=\emptyset$. 

Now let $\Theta:=\{\theta_n^{\Theta}\}_{n\in \N}\in \Psi_0$. For each $n\in \N$, $(cos\theta_n^{\Theta})^{-1}>1$ so that there is a $p_n\in \N$ satisfying $(cos\theta_n^{\Theta})^{-p_n}>2.$ Define

\begin{equation}
\theta_n:=
\begin{cases}
\theta_1^{\Theta} \ \ 1\leq n \leq P_1, \cr
\theta_i^{\Theta} \ \ \sum_{j=1}^{i-1}P_j<n\leq \sum_{j=1}^iP_j.
\end{cases}
\end{equation}
We observe that $\{\theta_n\}$ is non-increasing,
$$\sup_{n\in \N}\theta_n=\sup_{n\in \N}\theta_n^{\Theta}\leq \frac{\pi}{24},\ \ \lim_{n\rightarrow \infty}\theta_n=\lim_{\rightarrow \infty}\theta_n^{\Theta}=0,$$
and calculate
$$\prod_{n=1}^{\infty}(cos\theta_n)^{-1}=\prod_{n=1}^{\infty}\prod_{j=1}^{p_n}(cos\theta_n^{\Theta})^{-1}>\prod_{n=1}^{\infty}2=\infty.$$
We deduce that $\Theta\in \Psi_{\infty}$ and thus $\Theta\in \Psi_0\cap \Psi_{\infty}$.

Take now $\Phi:=\{\theta_n^{\Phi}\}_{n\in \N} \in \Psi_0\cap \Psi_{\infty}$
For each $n\in \N$, define $M_n$ to be the smallest even integer satisfying $M_n\theta_n^{\Phi}>2\pi$. Define $\alpha_n:=101010...10$ to be the base 2 string of digits of length $M_n$ alternating between 1 and 0. Set
$$P_1=:N_0(\alpha_1,M_0,2^{-1})$$
where $N_0(\cdot ,\cdot ,\cdot )$ is as defined in Lemma $\ref{RN0andc}$. More generally, for $n\in \N$, $n>1$, set
$$P_n:=N_0(\alpha_n,P_{n-1},1-2^{-n})$$
and define
\begin{equation}
\phi_n:=
\begin{cases}
\theta_1^{\Phi} \ \ 1\leq n\leq P_1 \cr
\theta_i^{\Phi} \ \ P_{i-1}< n\leq P_i,
\end{cases}
\end{equation}
and $\Phi_0:=\{\phi_n\}_{n\in\N}$. We show that $\Theta \in \Psi_0\cap \Psi_{\infty}\cap \Psi_R^{1}$.

As $\Theta \in \Psi_0\cap \Psi_{\infty}$, it is clear, from its definition, that $\Phi_0$ is a non-increasing sequence of real numbers smaller than $\pi/24$ with
$$\sup_{n\in \N}\phi_n=\sup_{n\in \N}\theta_n^{\Phi}<\frac{\pi}{24},\ \  \lim_{n\rightarrow \infty}\phi_n=\lim_{n\rightarrow \infty}\theta_n^{\Phi}=0$$
and that, since $\phi_n\geq \theta_n^{\Phi}$ for each $n\in \N$,
$$\prod_{n=1}^n(cos\phi_n)^{-1}\geq \prod_{n=1}^{\infty}(cos\theta_n^{\Phi})^{-1}=\infty$$
so that $\Phi\in \Psi_{\infty}$.

Now, for each $n\in \N$ let $R_n:=R(P_{n-1},\alpha_n,P_n)$ and define 
$$Q:=\bigcap_{k=1}^{\infty}\bigcup_{n=k}^{\infty}R_n.$$
By the selection of $P_n$ for $n\in \N$, we infer from Lemma $\ref{RN0andc}$ that $\Hm{1}(R_n)\geq 1-2^{-n}$ and thus 
\begin{equation}\label{3cap1}
\Hm{1}(Q)=1.
\end{equation}
Now, let $x\in Q$ and $m_0\in \N$. Then there is an $n_0\in \N$ with $P_{n-1}>m_0$ for all $n\geq n_0$. Since $x\in Q$, $x\in R_{n_1}$ for some $n_1\geq n_0$. By the selection of $P_n$ and $R_n$ for $n\in \N$, it follows that there is an $m_1>m_0$ with 
$$x=x_0.x_1...x_{m_1}10...10x_{m_1+M_n+1}... ,$$
where the central $1010...10$ is $\alpha_n$, and $\phi_i=\theta_n^{\Phi}$ for $m_1+1\leq i\leq m_1+M_n$. By Corollary $\ref{absthetasum}$ we now have
$$\left|\sum_{m=m_1+1}^{m_1+M_n}\vt_{m,i(m,\calF_{\Phi_0}(x))}\right|=\sum_{m=m_1+1}^{m_1+M_n}\theta_m^{\Phi}=M_n\theta_n^{\Phi}>2\pi$$
and thus, together with ($\ref{3cap1}$), $\Theta \in \Psi_R^1$.
\end{proof}	

\begin{Lemma}\label{pcsubp0pinfty}
$$\Psi_c\subset \Psi_{\infty}\cap \Psi_R^1.$$
\end{Lemma}
\begin{proof}
Let $\Theta:=\{\theta_n=\theta\}_{n\in\N}$. As $(cos\theta)^{-1}=:c>1$
$$\prod_{n=1}^{\infty}(cos\theta_n)^{-1}=\lim_{n\rightarrow \infty}c^n=\infty$$
so that $\Theta \in \Psi_{\infty}.$

Let $M$ be an even integer satisfying $M\theta>2\pi$. Set $\alpha=a_1...a_M$ to be the base 2 string of $M$ digits satisfying
\begin{equation}
a_i=
\begin{cases}
1 \ \ i\hbox{ odd} \cr
0 \ \ i \hbox{ even}.
\end{cases}
\end{equation}
By the Normal Number Theorem, the sequence $\alpha$ occurs infinitely often in $\Hm{1}$-almost all $x\in [0,1]$. In particular for $\Hm{1}$-almost all $x\in [0,1]$ and all $n\in\N$ there is an $m>n$ such that 
$$x=x_0.x_1...x_n...x_ma_1a_2...a_Mx_{m+M+1}...$$
so that, by Corollary $\ref{absthetasum}$,
$$\left|\sum_{n=m+1}^{m+M}\vt_{n,i(n,\calF_{\Theta}(x))}\right|=\sum_{n=m+1}^{m+M}\theta_n=M\theta>2\pi.$$
That is, $x\in R(\Theta)$. It follows that $\Theta\in \Psi_R^1$.
\end{proof}
\begin{Remark}
As mentioned in Remark $\ref{kochrectrem}$, that $\At$ is purely unrectifiable for $\Theta \in \Psi_c$ follows from $\At$ being a Koch curve. However, as a Corollary to the above Lemma we give a short direct proof of the pure unrectifiability of certain subsets of $\At$ for $\Theta \in \Psi_c$, which is sufficient for our classification.
\end{Remark}
 \begin{Corollary}\label{psicunrect}
Let $\Theta\in \Psi_c$, then there exists a compact purely unrectifiable subset $\Gamma \subset\At$.
\end{Corollary}
\begin{proof}
Let $\Theta \in \Psi_c$. By Lemma $\ref{pcsubp0pinfty}$ $\Psi_c\subset \Psi_R^1$ and thus $\Hm{1}(R(\Theta))=1$. It follows that there is an open set, $U\supset [0,1]\sim R(\Theta)$, with $\Hm{1}(U)<1$ and thus a compact set, $A\subset R(\Theta)$, with $\Hm{1}(A)>0$. The result now follows from Lemma $\ref{unrect}$ with $\Gamma:=\Ft(A)$.
\end{proof}

The proof of the main existence theorem, Theorem $\ref{main2}$, now follows by combining Lemmata $\ref{lem2}$, $\ref{lem10}$, $\ref{unrect}$, and $\ref{tripleintersection}$ and Proposition $\ref{HF1infinite}$. In proving Theorem $\ref{main2}$ we also complete our classification, Theorem $\ref{classification}$. We present a formal proof of Theorem $\ref{classification}$ directly following the proof of Theorem $\ref{main2}$.
\begin{Theorem2}
There is a compact set $A\subset \R^2$ with the following properties
\begin{enumerate}[(i)]
\item $A\in R(w,\rho,\delta;1)$,
\item $dim_{\Hm{}}A=1$,
\item $A$ is neither weakly nor strongly locally $\Hm{1}$ finite,
\item $A$ is not $\Hm{1}$-$\sigma$-finite and
\item $A$ is purely unrectifiable.
\end{enumerate}
\end{Theorem2}
\begin{proof}
From Lemma $\ref{tripleintersection}$ we may choose $\Theta\in \Psi_0\cap\psi_{\infty}\cap\Psi_{R}^1.$ Since $\Theta\in \Psi_R^1$, $\Hm{1}(R(\Theta))$=1. We deduce that $\Hm{1}(D \cup([0,1]\sim R(\Theta)))=0$ and thus that there exists an open set $U\supset D\cup([0,1]\sim R(\Theta))$ with $\Hm{1}(U)<1$. Define
$$A_0:=[0,1]\sim U\subset R(\Theta)$$
and
$$A:=\Ft(A_0).$$
We note that $\Hm{1}(A)>0$ and that $A$ is compact. 

Since $E(\Theta)=\Ft(D)$ we see that $A\subset \At\sim E(\Theta)$ so that, since $\Theta \in \Psi_0$, (i) follows from Lemma $\ref{lem10}$. (ii) now follows from Lemma $\ref{lem2}$.

Since $\Theta\in\Psi_{\infty}$, it follows from Proposition $\ref{HF1infinite}$ that 
$$\Hm{1}(\Ft^{-1}(B))=0$$ 
for any $B\subset A$ with $\Hm{1}(B)<\infty$. We deduce that $\Hm{1}(\Ft^{-1}(S))=0$ for any $\Hm{1}$-$\sigma$-finite set $S\subset A$, and therefore, since $$\Hm{1}(\Ft^{-1}(A))=\Hm{1}(A)>0,$$ 
$A$ is not $\Hm{1}$-$\sigma$-finite which proves (iv).

Should $\rho_x>0$ exist for each $x\in A$ with 
$$\Hm{1}(B_{\rho_x}(x)\cap A)<\infty,$$ 
then, since $A$ is compact, there exists a collection of finitely many balls, $\{B_{\rho_i}(x_i)\}_{i=1}^P$ with $x_i\in A$ for $i\in \{1,...,P\}$, satisfying
$$\Hm{1}(B_{\rho_i}(x_i)\cap A)<\infty \ \ \hbox{ and } A=\bigcup_{i=1}^P(B_{\rho_i}(x_i)\cap A).$$
Since $A$ is not $\Hm{1}$-$\sigma$-finite this is impossible. We deduce the existence of $x\in A$ with 
$$\Hm{1}(B_{\rho}(x)\cap A)=\infty$$
 for all $\rho>0$, showing (iii).
 
 Finally, since $A\subset R(\Theta)$ and $\Hm{1}(A)>0$, (v) follows from Lemma $\ref{unrect}$.
\end{proof}

\begin{Theorem1}
The properties defined in Definition $\ref{defa}$ satisfy the classification given in the table below with respect to the questions given in Question $\ref{ques1}$.
\begin{equation} 
\begin{tabular}[h]{lccc}

\hbox{Property} &  & \hbox{Question} &     \nonumber \\
\hline
& & &  \nonumber \\
& (1) & (2) & (3) \nonumber \\
&  &  \hbox{ (a), (b)}&     \nonumber \\ \hline
& & &  \nonumber \\
$w j$ & \hbox{No} & \hbox{No, No}  & No  \nonumber \\
$w\rho j$  & \hbox{No} & \hbox{No, No} & No  \nonumber \\
$w\rho_0 j$ & \hbox{No}  & \hbox{No, No} & No  \nonumber \\
$w\delta j$ & \hbox{Yes} & \hbox{No, No}& No  \nonumber \\
$w\rho\delta j$ & \hbox{Yes} &  \hbox{No, No}& No  \nonumber \\
$w\rho_0\delta j$ & \hbox{Yes} & \hbox{Yes, Yes} & Yes  \nonumber \\
$s j$ & \hbox{Yes}  & \hbox{No, No} & Yes \nonumber \\
$s\rho j$ & \hbox{Yes} & \hbox{Yes, No}& Yes  \nonumber \\
$s\rho_0 j$ & \hbox{Yes} &  \hbox{Yes, Yes}& Yes  \nonumber \\
$s \delta j$ & \hbox{Yes} & \hbox{No, No} & Yes  \nonumber \\
$s \rho\delta j$& \hbox{Yes}  & \hbox{Yes, No} & Yes \nonumber \\
$s\rho_0\delta j$ & \hbox{Yes} & \hbox{Yes, Yes} & Yes  \nonumber \\ \hline

\end{tabular}
\end{equation}
\end{Theorem1} 
\begin{proof} 
The positive answers follow from Corollary $\ref{cor1}$ and Lemma $\ref{propertyiv}$. 

That $\calN\in R(s,\rho,\delta;1)$ and $\Lambda\in R(s,\emptyset, \delta; 1)$ follow from Proposition $\ref{easyproperties}$. That $\calN$ and $\Lambda$  additionally satisfy the measure properties required to answer questions $s\rho\delta j$ (2) (b) and $s\delta j$ (2)((a) and (b)) with no respectively follows from Lemma $\ref{easybadprops}$.

By Lemma $\ref{lem10}$ (ii)
$$R(w,\rho_0,\emptyset ; 1) \cap \calA(\Psi_c) \supset R(w,\rho_0,\emptyset ; 1) \cap \calA(\Psi_c,\eta(\delta))\not= \emptyset$$
for all $\delta >0$ so that, by Lemma $\ref{lem12}$, the answers to $w\rho_0 j$ (1), (2)((a) and (b)), and (3) are no.

We infer from Theorem $\ref{main2}$ that the answers to $w\rho \delta j$ (2)((a) and (b)) and (3) are no.

As
\begin{eqnarray}
R(s,\rho,\delta;1)&\subset & R(s,\rho,\emptyset; 1), \nonumber \\
R(s,\emptyset,\delta;1)&\subset & R(s,\emptyset,\emptyset; 1), \nonumber \\
R(w,\rho,\delta;1)&\subset & R(w,\emptyset,\emptyset; 1)  \ \ \hbox{ and }\nonumber \\
R(w,\rho_0,\emptyset;1) & \subset & R(w,\rho,\emptyset; 1) \subset R(w,\emptyset,\emptyset; 1), \nonumber
\end{eqnarray}
we deduce from the preceding three paragraphs that the answers to the remaining questions are no.
\end{proof}

\begin{Remark}
We have only given counter examples for $j=1$ showing that the questions answered with no cannot be answered with yes for all $j$. However, for any given one of the counter examples, say $A$, considered above, $A$ can be extended to be a counter example in dimension $j$ by taking $A\times [0,1]^{j-1}$.
\end{Remark}

\section{Relationship with singular sets}
As has been mentioned, this investigation was made with application to the singular sets of geometric flows in mind. In this final section we observe a question arising from this consideration. Of the defined approximation properties, we note that, most particularly, the $w\rho\delta j$ property has been shown (Simon \cite{simon2}) to be applicable to the singular sets of minimal surfaces. The $w\rho\delta j$ property also has the most interesting classification in the sense of the preceding sections. This, in the sense that sets in $R(w,\rho,\delta ;j)$ must be $j$-dimensional, need not have any other regularity properties, but that the examples of `poor behaviour' appear to be necessarily complex.

The facts that the $w\delta j$ property has the same classification as the $w\rho\delta j$ property, and that $\Lambda$ satisfies the $w\delta j$ property lead us to ask whether a `simpler' set in $R(w,\rho,\delta ;j)$ that does not  have locally finite measure exists, or whether such sets really are necessarily complex? We answer this question for general dimensions below by proving that such sets are indeed necessarily complicated. The sets must be complicated in the sense that no point of infinite density in a set $A\in R(w,\rho,\delta ;j)$, say $y$, may be an element of a piece of Lipschitz graph in $A$. That is, $A$ must accumulate infinite measure around $y$ without any part of $A$ being differentiable. This property could certainly be a stepping stone in showing the regularity of the singular set of geometric flows. 

After defining density, the proof of the complexity of counter examples to question $w\rho\delta j$ (2) can be proven directly.

\begin{Definition}\label{lowerdensity}
For a subset $A \subset \R^n$, $y\in \R^n$, and $m \leq n$ we define the lower $m$-dimensional density of $A$ at $y$ to be 
\[\Theta^m_*(A,\Hm{m},y):= \liminf_{\rho \rightarrow \infty}\frac{\Hm{m}(A\cap B_{\rho}(y))}{\omega_m\rho^m}\]
where $\omega_m$ denotes the $\Hm{m}$ measure of the unit $m$-ball.
\end{Definition}

\begin{Theorem}\label{thm52}
Let $A\subset \R^n$ and suppose that $y\in A$, $\rho_1>0$, $G_y\in G(n,m)$, and that a Lipschitz function $u:G_y \rightarrow G_y^{\perp}$ exist with 
$$\Theta^m_*(A ,\Hm{m}, y)=\infty, \ \ y\in U, \ \ \hbox{ and } \ \ \overline{B_{\rho_1}(y) \cap A \cap U} = U \cap \overline{B_{\rho_1}(y)},$$
where $U:=graph(u)$. Then $A\not\in R(w,\rho,\delta;m)$.
\end{Theorem}
\begin{proof}
Suppose $A\in R(w,\rho,\delta;m)$. Take $\tilde{\rho_y} \leq \rho_1$ such that for all $\rho\leq \tilde{\rho_y}$ 
\[\frac{\Hm{m}(B_{\rho}(y) \cap A)}{w_m\rho^m}\geq2\sqrt{1+(Lipu)^2}.\]
Take $\delta < \min\{(24Lipu)^{-1},1/16\}$. Further, take $\rho_y \leq \tilde{\rho_y}$ to be the radius, dependent on $y$, given by the definition of the $w\delta j$-property with respect to $\delta$ and $j=m$.

Suppose now that for some $x\in A \cap B_{\rho_y}(y)$ and $\rho \in (0,\rho_y]$ 
\[U \cap (B_{\rho}(x) \sim \overline{L_{x,\rho}^{\delta \rho}})\not= \emptyset.\]
Then, there exists a $p \in U$ and $r_p >0$ such that 
\[B_{r_p}(p) \subset B_{\rho}(x) \sim \overline{L_{x,\rho}^{\delta \rho}}.\]
Since
\[\overline{A \cap U \cap B_{\rho_1}(y)} = U \cap \overline{B_{\rho_1}(y)},\]
it follows that there exists a $w\in B_{r_p}(p)\cap A$ which is a contradiction to $A\in R(w,\emptyset, \delta;m)\subset R(w,\rho,\delta;m)$.

We can therefore assume that for all $x\in A$, $\rho \in (0,\rho_y]$ and $P\subseteq U$
\[B_{\rho}(x) \cap P \subset B_{\rho}(x) \cap \overline{L_{x,\rho}^{\delta\rho}}.\]
Note now that 
\begin{eqnarray}
\Hm{m}(B_{\delta \rho_y}(y)\cap U)  \leq  \Hm{m}(u(\pi_{G_y}(B_{\delta\rho_y}(y)))) 
 \leq  w_m\rho_y^m\delta^m\sqrt{1+(Lipu)^2} 
 \leq  \frac{\Hm{m}(A \cap B_{\delta \rho_y}(y))}{2}. \nonumber
\end{eqnarray}
It follows that there exists
\[x\in (A \cap B_{\delta \rho_y}(y)) \sim U.\]
Suppose now that $u(\pi_{G_y}(x))\not\in B_{\rho_y /3}(y)$. Then 
\begin{eqnarray}
|u(\pi_{G_y}(x))-u(\pi_{G_y}(y))|  >  \frac{\rho_y}{3} 
 > \frac{|\pi_{G_y}(x) - \pi_{G_y}(y)|}{3\delta} 
 >  Lipu|\pi_{G_y}(x) - \pi_{G_y}(y)|. \nonumber
\end{eqnarray}
This contradiction ensures that $u(\pi_{G_y}(x))\in B_{\rho_y /3}(y)$. 

We now write $z:=u(\pi_{G_y}(x))$. Further, by otherwise shifting $A$, we can assume that $x=0$. Consider now $B_{2d(x,z)}(x)$ and note that $x\in B_{\rho_y}(y)$ and $2d(x,z)<2\rho_y(\delta + 1/3) < \rho_y$. 

Writing $\rho_x:=2d(x,z)$ there therefore exists an $m$-dimensional plane $L_{x,\rho_x} \in G(n,m)$ such that 
\[L_{x,\rho_x}^{\delta \rho_x}\cap B_{\rho_x}(x) \supset A \cap B_{\rho_x}(x).\]
We argue as above to find that this implies
\begin{equation}\label{uxball}
\overline{L_{x,\rho_x}^{\delta \rho_x}\cap B_{\rho_x}(x)} \supset U \cap \overline{B_{\rho_x}(x)}.
\end{equation}
Since $x,z \in \overline{A} \cap B_{\rho_x}(x)$, $x,z \in \pi_{G_y}^{-1}(\pi_{G_y}(x))=\pi_{G_{y}}^{-1}(0)$ and $|x-z|=\rho_x/2$ it follows that a unit vector $l \in L_{x,\rho_x}$ exists satisfying
\begin{equation}
|\langle l,v \rangle |\leq 2\delta
\label{bigthm1}
\end{equation}
for all unit vectors $v \in G_y$.

Now, let $\{l,l_2, ..., l_m\}$ be an orthonormal basis for $L_{x,\rho_x}$ and note that $L_y:=$ span$(\{\pi_{G_y}(l_i)\}_{i=2}^m)$ is a subspace of $G_y$ with dimension no more than $m-1$. Further, we write 
\[L_i:=\{t\pi_{G_y}(l_i):t\in \R\}.\]
As $G_y$ is an $m$-dimensional plane, we can find some $v\in G_y$ such that $\la v,w\ra=0$ for all $w\in L_y$. By rotation, we assume that $w=e_1$ and $l_i=c_ie_i$ for $i=2,...,m$, where $c_i \in [0,1]$ and $\{e_i\}_{i=1}^n$ denotes the canonical orthonormal basis in $\R^n$.

If $\pi_{G_y}(L_{x,\rho_x})\subset L_y$ we see that
\[\pi_{G_y}(L_{x,\rho_x}^{\delta \rho_x}\cap\overline{B_{\rho_x}(x)}) \subset [-\delta\rho_x , \delta \rho_x]\times\prod_{i=1}^mL_i.\]
Otherwise $\pi_{G_y}(l)\not=\{0\}$ in which case we can take $w:=\frac{\pi_{G_y}(l)}{|\pi_{G_y}(l)|}$ (again, by rotation we assume $w=e_1$). Moreover, in this case, by ($\ref{bigthm1}$) we see that
\[\pi_{G_y}(L_{x,\rho_x}^{\delta \rho_x}\cap\overline{B_{\rho_x}(x)}) \subset [-6\delta\rho_x , 6\delta \rho_x]\times\prod_{i=1}^m L_i.\]
In either case, therefore, we now have
\begin{equation}\label{strip}
\pi_{G_y}(L_{x,\rho_x}^{\delta \rho_x}\cap\overline{B_{\rho_x}(x)}) \subset [-6\delta\rho_x , 6\delta \rho_x]\times\prod_{i=1}^m L_i.
\end{equation}
We now consider $u|_{\R_1}$. Note that $u|_{\R_1}$ is a Lipschitz function with Lip$u|_{\R_1} \leq $ Lip$u$.

Noting that, by ($\ref{uxball}$) and ($\ref{strip}$),
\begin{enumerate}
\item $u|_{\R_1}(-7\delta \rho_x)\not\in B_{\rho_x}(x)$,
\item $u|_{\R_1}(0)=u|_{\R_1}(\pi_{G_y}(x))=u(\pi_{G_y}(x))=z\in B_{\rho_x}(x)$,
\item $U_1:=graph(u|_{\R_1})$ is connected, and
\item $\pi_{G_y}(U_1\cap \partial B_{\rho_x}(x)) \subset \pi_{G_y}(\overline{U \cap B_{\rho_x}(x)})\cap \R_1 \subset [-6\delta \rho_x , 6\delta \rho_x]$,
\end{enumerate}
it follows that there exists a $\tilde{z} \in [-6\delta \rho_x , 0)$ such that $u|_{\R_1}(\tilde{z})\in \partial B_{\rho_x}(x)$. Thus $|\tilde{z}-\pi_{G_y}(z)|\leq 6\delta \rho_x$ and hence
\begin{eqnarray}
|u(\tilde{z})-u(\pi_{G_y}(z))|  \geq  d\left(\pi_{G_y^{\perp}}\left(\overline{\partial B_{\rho_x}(x) \cap L_{x,\rho_x}^{\delta, \rho}}\right)\cap \R_1,z\right) 
 >  \frac{\rho_x}{4} 
 \geq  \frac{\rho_x |\tilde{z} - \pi_{G_y}(z)|}{24\delta \rho_x} 
 >  Lipu|\tilde{z} - \pi_{G_y}(z)|. \nonumber
\end{eqnarray}
This contradiction implies that the assumption, $A\in R(w,\rho,\delta;m)$, is false, completing the proof.
\end{proof}

\end{document}